\documentclass{amsart}
\usepackage{graphicx} 
\usepackage{hyperref}
\usepackage{amsmath, amsthm, amssymb, amscd}
\usepackage{enumerate}
\usepackage{hyperref}
\usepackage{verbatim}
\usepackage{cite}
\usepackage{color}
\usepackage{esint}
\usepackage{tikz-cd}
\usepackage{comment}

\DeclareMathOperator{\diam}{diam}

\theoremstyle{plain}
\newtheorem{theorem}{Theorem}

\newtheorem{corollary}[theorem]{Corollary}

\newtheorem{claim}[theorem]{Claim}

\newtheorem{lemma}[theorem]{Lemma}

\newtheorem{proposition}[theorem]{Proposition}
\newtheorem{conjecture}[theorem]{Conjecture}

\theoremstyle{definition}

\newtheorem{definition}[theorem]{Definition}
\newtheorem{remark}[theorem]{Remark}
\newtheorem{question}[theorem]{Question}

\numberwithin{equation}{section}
\numberwithin{theorem}{section}

\newcommand{\Hdim}{\text{dim}_H}
\newcommand{\Sdim}{\text{dim}_s}

\newcommand{\RR}{\mathbb{R}}
\newcommand{\HH}{\mathcal{H}}

\newcommand{\ZZ}{\mathbb{Z}}

\newcommand{\N}{\mathbb{N}}
\newcommand{\cF}{\mathcal{F}}
\newcommand{\cC}{\mathcal{C}}
\newcommand{\R}{\mathbb{R}}
\newcommand{\Lip}{\textnormal{Lip}}
\newcommand{\QS}{\textnormal{QS}}

\allowdisplaybreaks

\title{Analytically one-dimensional planes and the Combinatorial Loewner Property}
\author{Guy C. David}
\email{gcdavid@bsu.edu}
\address{Department of Mathematical Sciences, Ball State University, Muncie, IN 47306}

\author{Sylvester Eriksson-Bique}
\email{sylvester.d.eriksson-bique@jyu.fi}
\address{Department of Mathematics and Statistics
P.O. Box 35
FI-40014 University of Jyväskylä}

\thanks{The first author was partially supported by National Science Foundation grant DMS-2054004. The second author was supported by the Research Council of Finland grants 354241 and 356861. One of the crucial auxiliary results in this paper, Theorem \ref{thm:LQ}, arose from earlier unpublished work of the first author and Jeff Cheeger. We thank Cheeger for these discussions and for the permission to include the result in the present paper. We also thank Riku Anttila, Mario Bonk and Mathav Murugan for helpful discussions.}
\subjclass[2020]{30L10, 20F65, 51F99, 53C23, 28A78}

\date{\today}

\begin{document}

\begin{abstract}
It is a major problem in analysis on metric spaces to understand when a metric space is quasisymmetric to a space with strong analytic structure, a so-called Loewner space. A conjecture of Kleiner, recently disproven by Anttila and the second author, proposes a combinatorial sufficient condition. The counterexamples constructed are all topologically one dimensional, and the sufficiency of Kleiner's condition remains open for most other examples.

A separate question of Kleiner and Schioppa, apparently unrelated to the problem above, asks about the existence of ``analytically $1$-dimensional planes'': metric measure spaces quasisymmetric to the Euclidean plane but supporting a $1$-dimensional analytic structure in the sense of Cheeger.

In this paper, we construct an example for which the conclusion of Kleiner's conjecture is not known to hold. We show that \emph{either} this conclusion fails in our example \emph{or} there exists an ``analytically $1$-dimensional plane''. Thus, our construction either yields a new counterexample to Kleiner's conjecture, different in kind from those of Anttila and the second author, or a resolution to the problem of Kleiner--Schioppa.
\end{abstract}

\maketitle

\section{Introduction}

\subsection{The Combinatorial Loewner Property and Kleiner's conjecture}

Quasiconformal and quasisymmetric mappings are generalizations of classical conformal mappings to the setting of non-smooth metric spaces. Roughly speaking, quasisymmetric mappings are those that preserve the ``shapes'' of objects in a space, while possibly distorting their sizes. They are fundamental objects of study in analysis on metric spaces, complex analysis, and geometric group theory.

A basic question in these situations is: Given a metric space, is it quasisymmetrically equivalent to a space with good analytic structure? This is, in some sense, an optimization problem within the ``conformal'' class of the given space.

The good analytic structure we search for is the so-called \emph{Loewner condition} of Heinonen--Koskela \cite{HK}. This is a property that guarantees the existence of many rectifiable curves connecting pairs of continua in the space. Much work over the past thirty years has shown that this property further implies a very strong theory of calculus for quasiconformal and Lipschitz mappings; see \cite{HK, Ch, HajlaszKoskela, BateLi, HKST} for a sample.

Furthermore, in certain contexts, even just knowing the existence of a Loewner metric in the  quasisymmetry class of a given space has strong analytic consequences. For instance, Bonk and Kleiner \cite{BK04} show that the existence of such a metric in the case of certain boundaries is equivalent to Cannon's conjecture concerning group actions on hyperbolic $3$-space. This is also connected to long-standing questions about ``conformal dimension''. We refer the reader to \cite{KleinerICM, BonkICM, MT,CS, AEBa} for related discussion.

Thus, it has been a major problem to understand when a given metric space is quasisymmetrically equivalent to a Loewner space. To this end, Kleiner \cite{KleinerICM} (see also \cite{BourK}) introduced the \emph{Combinatorial Loewner Property}. Unlike the Loewner property, the Combinatorial Loewner Property is defined by looking at discrete approximations of the space, and is a quasisymmetric invariant. Kleiner proposed, under a weak self-similarity assumption known as ``approximate self-similarity'' that is satisfied by boundaries of Gromov hyperbolic groups, that the Combinatorial Loewner Property may be sufficient to find a Loewner metric in the quasisymmetry class of a space:

\begin{conjecture}[Kleiner \cite{KleinerICM}, Bourdon--Kleiner \cite{BourK}]\label{conj:Kleiner}
Suppose $Z$ is an approximately self-similar metric space. If $Z$ satisfies the Combinatorial Loewner Property, then $Z$ is quasisymmetric to an Ahlfors regular Loewner space.
\end{conjecture}

However, very recently, Anttila and the second author \cite{AEBa} have shown that Conjecture \ref{conj:Kleiner} \textbf{fails} in general by constructing certain (topologically) one-dimensional counterexamples. We note that it is still an open question whether even the classical Sierpi\'nski carpet is quasisymmetric to a Loewner space, though interesting examples of planar carpets for which the conclusion of Conjecture \ref{conj:Kleiner} does hold can be found in \cite{CS}.

The techniques of \cite{AEBa} appear to be restricted to the one-dimensional setting, and (as noted above) it is of substantial interest to see how Conjecture \ref{conj:Kleiner} fares in higher-dimensional settings.   Additional background and motivation for Conjecture \ref{conj:Kleiner} can also be found in \cite{AEBa}.

In this paper, we study a specific two-dimensional Combinatorially Loewner space where the techniques of \cite{AEBa} do not apply and for which we do not know if the conclusion of Conjecture \ref{conj:Kleiner} holds.  We show that if Conjecture \ref{conj:Kleiner} holds for this particular example, then this would resolve another well-known problem in analysis on metric spaces. We now describe that problem. 

\subsection{Differentiability on metric measure spaces and the question of ``analytically one-dimensional planes''}

Fundamental work of Cheeger \cite{Ch} provides a type of differentiable structure for certain non-smooth metric measure spaces, the so-called \emph{PI (Poincar\'e Inequality) spaces}. On such spaces, Lipschitz functions are ``differentiable'' almost everywhere, i.e., a version of Rademacher's theorem holds, and a remarkable amount of first-order calculus can be done in these non-smooth settings. (More detailed definitions are in Section \ref{sec:prelims}.)

The differentiable structure on a PI space comes equipped with a notion of dimension, the \emph{analytic dimension} of the space, which should be thought of as the dimension of the cotangent bundle associated to the differentiable structure. On the standard $\RR^n$, the analytic dimension is $n$, but on other spaces the analytic dimension differs from the Hausdorff and topological dimensions, and is more subtly related to the geometry of the space.

A basic question concerns what happens when the underlying metric space is a topological manifold; since the questions are infinitesimal it makes sense to impose this condition on the Gromov--Hausdorff blowups of the space. The following is the simplest open question in this regime, which has resisted solution for nearly a decade. 

\begin{question}[Kleiner--Schioppa \cite{KleinerSchioppa}]\label{q:1dimplane}
Suppose $(X,d,\mu)$ is a PI space. Assume that all Gromov--Hausdorff limits of pointed rescalings of $X$ are homeomorphic to $\mathbb{R}^2$. Can the analytic dimension of the space be $1$?
\end{question}
The terminology here is explained more carefully in Section \ref{sec:prelims}. We remark that if $(X,d,\mu)$ is a PI space that is quasisymmetric to the plane $\RR^2$, then it will satisfy the assumptions of Question \ref{q:1dimplane}, and to our knowledge the question of whether such a space can have analytic dimension $1$ is open even in this setting. We colloquially refer to hypothetical examples that might positively answer Question \ref{q:1dimplane} as ``analytically $1$-dimensional planes'', and refer to the introduction of \cite{DavidKleiner} for more background on this and closely related questions.

\subsection{New results}
Our main result concerns a specific metric space $(X,d_0)$, which we call the ``pillow space''. We give an informal description of the space here, and a more careful one in subsection \ref{subsec:pillowspace}.

Let $X_0$ be the unit square in $\RR^2$. Divide it into nine congruent sub-squares. Now add in another copy of the middle square that is glued along its boundary to the original middle square. This is the first stage $X_1$, a ``pillow'' built of ten $\frac13\times\frac13$ squares.

Now repeat this construction at a smaller scale on each of those ten small squares: subdivide each of them into nine congruent pieces, and double each middle square along the boundary. This yields $X_2$.

Repeat this process indefinitely, obtaining a sequence of spaces $X_n$. Equip each $X_n$ with the geodesic metric for which each of its $10^n$ $3^{-n}\times 3^{-n}$ squares is isometric to its Euclidean counterpart.  There are then $1$-Lipschitz transition maps $X_{n+1}\rightarrow X_n$ given by ``collapsing'' the pillows of scale $n+1$, and so we may take an inverse limit to obtain a self-similar geodesic metric space $(X, d_0)$. (We reiterate that this is only an informal description, and a precise definition is in Definition \ref{def:pillowspace} below.) 

The link to Conjecture \ref{conj:Kleiner} is via the following result of Anttila and the second author \cite[Theorem 1.1, Proposition 1.7, Example 4.3]{AEBb}. 
\begin{theorem}[Anttila--Eriksson-Bique]\label{thm:AEB}
$(X,d_0)$ is approximately self-similar and has the Combinatorial Loewner Property.
\end{theorem}

The main result of the present paper is the following. (Definitions of the terms involved can be found in Section \ref{sec:prelims}.)

\begin{theorem}\label{thm:main}
Suppose that the pillow space $(X,d_0)$ is quasisymmetric to an Ahlfors $Q$-regular $Q$-Loewner space. Then there exists a metric $\rho$ and a measure $\mu$ on the unit square $[0,1]^2$ such that
\begin{enumerate}[(i)]
\item $([0,1]^2,\rho)$ is quasisymmetric to the standard unit square $([0,1]^2,|\cdot|)$.
\item $([0,1]^2,\rho,\mu)$ is a PI space.
\item $([0,1]^2,\rho,\mu)$ has analytic dimension $1$.
\end{enumerate}
\end{theorem}

By an additional argument, we can prove a similar result about planes rather than squares, more directly addressing Question \ref{q:1dimplane}. As remarked above, standard compactness results for quasisymmetric mappings imply that a PI space that is quasisymmetric to the standard plane will satisfy the assumption of Question \ref{q:1dimplane}.

\begin{corollary}\label{cor:1dimplane}
Suppose that the pillow space $(X,d_0)$ is quasisymmetric to an Ahlfors $Q$-regular $Q$-Loewner space. Then there is a PI space that is quasisymmetric to the Euclidean plane $\mathbb{R}^2$ and has analytic dimension $1$. Moreover, no Gromov--Hausdorff blowup of this PI space is bi-Lipschitz equivalent to $\RR^2$.
\end{corollary}

In other words, if the conclusion of Conjecture \ref{conj:Kleiner} holds for the pillow space, then the answer to Question \ref{q:1dimplane} is ``yes''. Thus, either the pillow space provides a new counterexample to Conjecture \ref{conj:Kleiner}, different in kind than those in \cite{AEBa}, or it provides an example solving Question \ref{q:1dimplane}. 

As noted in the introduction of \cite{DavidKleiner}, all known examples satisfying the assumptions of Question \ref{q:1dimplane} have the property that generic Gromov--Hausdorff blowups are bi-Lipschitz equivalent to planes, i.e., they are infinitesimally Euclidean. Thus, if the assumptions of Corollary \ref{cor:1dimplane} held, it would provide a truly new addition to our catalog of PI spaces.

\subsection{Organization of the paper and outline of the proof}
Let us describe the organization of the paper and the outline of the proof of Theorem \ref{thm:main}, highlighting some intermediate results of interest.

Section \ref{sec:prelims} of the paper contains preliminaries and background material. Section \ref{sec:cubical} describes a general framework for building self-similar spaces as limits of cubical complexes, and finally defines the pillow space rigorously as a specific instance of this framework in Definition \ref{def:pillowspace}. Section \ref{sec:cubical} also contains an outline of how Theorem \ref{thm:AEB} follows from \cite{AEBb}, since the framework there that is used to construct the pillow space is different from that of the present paper. Section \ref{sec:blowup} describes an ad hoc notion of ``blowup'' adapted to the self-similar structure of the pillow space, which we use frequently throughout the paper.

The first substantial result along the road to Theorem \ref{thm:main} concerns the symmetries of the pillow space $X$. A subgroup of the isometry group of $(X,d_0)$ can be generated by maps that ``flip'' all the pillows of a given scale, i.e., exchange their top and bottom sides. If $d$ is another metric on $X$ that is quasisymmetric to $d_0$, these mappings are \textit{a priori} no longer isometries of $d$; if they are, then we will say that $d$ is \emph{symmetric}. We show that it suffices to consider only symmetric metrics when searching for a Loewner structure on $X$:

\begin{proposition}\label{prop:SL}
Suppose that $(X,d_0)$ is quasisymmetric to a $Q$-Loewner space. Then there is a symmetric metric $d_{SL}$ on $X$ such that $(X,d_{SL})$ is $Q$-Loewner and quasisymmetric (by the identity map) to $(X,d_0)$.
\end{proposition}

In proving Theorem \ref{thm:main}, we may therefore assume that the pillow space admits a quasisymmetrically equivalent symmetric Loewner structure.

Next, we observe that $(X,d_0)$ is the union of squares isometric to $[0,1]^2$, each formed by considering only the top or bottom of each ``pillow'' in $X$. We call these squares ``sheets''. If $d_{SL}$ is a symmetric Loewner metric on $X$, the sheets need not be isometric to $[0,1]^2$ anymore, but they are all isometric to each other.

We show that, given a symmetric Loewner metric on $X$ quasisymmetric to $d_0$, each of these (pairwise isometric) sheets admits a measure making it into a PI space.

\begin{corollary}\label{cor:sheetPIintro}
Let $d_{SL}$ be a symmetric Loewner metric on $X$, quasisymmetric (by the identity) to $d_0$. Let $S$ be a sheet of $X$. Then there is a measure $\nu$ on $S$ making $(S, d_{SL}, \nu)$ a PI space.
\end{corollary}
Corollary \ref{cor:sheetPIintro} is a consequence of the more detailed Corollary \ref{cor:sheetPI} below. A significant step in the proof of Corollary \ref{cor:sheetPIintro}/\ref{cor:sheetPI} is the following general result, proven in the Appendix, which we consider to be of independent interest. This theorem first arose in earlier unpublished work of the first author and Jeff Cheeger.

\begin{theorem}[Cheeger--David]\label{thm:LQ}
Let $(X,d,\mu)$ be any PI space, let $f\colon (X,d)\rightarrow (Y,\rho)$ be a Lipschitz quotient map, and let $\nu = f_{*}\mu$. If $\nu$ is finite on some ball, then $(Y,\rho,\nu)$  is a PI space.
\end{theorem}

Finally, in Section \ref{sec:mainproof} we show that the PI space from Corollary \ref{cor:sheetPIintro}, which is quasisymmetric to the standard square, has analytic dimension $1$ as follows: By a result of the second author and Kleiner (Theorem \ref{thm:davidkleiner}), the space has analytic dimension $1$ or $2$. Moreover, if the analytic dimension were $2$, that result would imply that each sheet has some $2$-rectifiable structure. After a blowup argument, we argue that this would force the hypothetical Loewner metric on $X$ to be bi-Lipschitz to the original metric $(X,d_0)$. However, it is not difficult to show that $(X,d_0)$ itself is not Loewner (see Lemma \ref{lem:notloewner}). Hence, the analytic dimension of the PI structure on the square sheets must be $1$.

\subsection{Remarks on constructions in other dimensions}
The pillow space $(X,d_0)$ appearing in Theorem \ref{thm:main} is topologically $2$-dimensional. One may consider a $1$-dimensional version of this construction: start with a line segment, double the middle third, and repeat inductively on each smaller edge. The limit yields a compact metric space, quite similar to examples in \cite{LangPlaut, CK_PI}. It is not difficult to see, however, that due to posessing local cut points this example is not Combinatorially Loewner, and hence cannot be quasisymmetric to an Ahlfors $Q$-regular $Q$-Loewner space. 

One could also investigate higher-dimensional pillow spaces: begin with a $d$-cube for some $d\geq 3$, subdvide into $3^d$ identical cubes, double the central one, and repeat. In this case, the limiting space \emph{would} be approximately self-similar and Combinatorially Loewner, i.e., the conclusion of Theorem \ref{thm:AEB} from \cite{AEBb} would still hold. (Here, one adapts the proof of Theorem \ref{thm:AEB} by using general cubical replacement rules in the sense of \cite{AEBb}.) In addition, many of the arguments of the present paper would go through, up until the use of Theorem \ref{thm:davidkleiner}, which is known only in dimension $2$. Thus, the methods of this paper could prove, for instance, the following weaker analog of Corollary \ref{cor:1dimplane} in higher dimensions: \emph{Suppose $d\geq 3$ and the $d$-dimensional analog of the pillow space is quasisymmetric to an Ahlfors $Q$-regular $Q$-Loewner space. Then there exists a PI space quasisymmetric to $\RR^d$ for which no Gromov--Hausdorff blowup is bi-Lipschitz to $\RR^d$.} Such a construction would also be new, even if the value of the analytic dimension would be unclear. Indeed, all known PI spaces quasisymmetric to $\mathbb{R}^d$ have the property that their generic Gromov--Hausdorff blowups are bi-Lipschitz equivalent to $\mathbb{R}^d$. (See the introduction of \cite{DavidKleiner} for a more general statement on our current state of knowledge.) Thus, this result would provide a truly new example in higher dimensions as well, if the assumption on the $d$-dimensional pillow space held.

In our view, these results provide a further strong impetus to understand the problem of finding Loewner metrics in the quasisymmetry classes of fractal spaces, in addition to the original motivations described above. 

\section{Notation, background, and preliminaries}\label{sec:prelims}
We use the notation $A\lesssim B$ to indicate that $A\leq cB$ for some implied constant $c$, which may change from line to line. If $A\lesssim B$ and $B\lesssim C$ then we write $A\approx B$. We also use the ``big O'' and ``little o'' notations in the standard way.

\subsection{Metrics and measures}

We use the notation $(X,d)$ to denote a metric space and $(X,d,\mu)$ to denote a metric measure space; our measures will always be Radon. Open balls in a metric space are denoted $B(x,r)=\{y: d(x,y)<r\}$. In some situations, to clarify the space or metric in which the ball is taken, we may write $B_X(x,r)$ or $B_d(x,r)$.

If $X$ is a metric space, $\epsilon>0$ and $A\subset X$, then $N_\epsilon(A)=\bigcup_{a\in A} B(a,\epsilon)$ is the $\epsilon$-neighborhood of $A$. We say that $A$ is $\epsilon$-dense in $B$, if $B\subset N_\epsilon(A)$.  

A metric space $(X,d)$ is called \emph{doubling} if every ball can be covered by a fixed number of balls of half the radius. A measure $\mu$ on a metric space is $C$-\emph{doubling} if $\mu(B(x,2r))\leq C\mu(B(x,r))$ for all points $x$ and radii $r$. A complete doubling metric space supporting a doubling measure must be a doubling metric space \cite[Ch. 13]{He}.

In addition to general measures, we also use the notions of Hausdorff measure, denoted $\HH^p$ in dimension $p$, and Hausdorff dimension, $\dim_H$, of a metric space; see \cite[Ch. 8]{He} for definitions.

A metric space $X$ is called \emph{Ahlfors $p$-regular} if there exists a constant $C>0$ such that, for all $x\in X$ and $r\in (0,2\diam(X)]$, we have
\[
C^{-1}r^p \leq \HH^p(B(x,r)) \leq C r^p.
\]

By a standard argument, this is equivalent to the property that there exists a Radon measure $\mu$ on $X$ for which
\[
C^{-1}r^p \leq \mu(B(x,r)) \leq C r^p \text{ for all }x\in X, r\in (0,2\diam(X)].
\]
This property is of course stronger than $\mu$ being a doubling measure.

\subsection{Mappings}
Let $(X,d_X)$ and $(Y,d_Y)$ be metric spaces, and $L>0$. A  mapping $f:X\rightarrow Y$ is $L$-\emph{Lipschitz} if 
$$ d_Y(f(x),f(y))\leq Ld_X(x,y) \text{ for all } x,y\in X$$
and $L$-\emph{bi-Lipschitz} if 
$$ L^{-1} d_X(x,y) \leq d_Y(f(x),f(y))\leq Ld_X(x,y) \text{ for all } x,y\in X.$$ If it is not necessary to emphasize the constant, we just say \emph{Lipschitz} or \emph{bi-Lipschitz}.

A slightly more elaborate notion is that of a quasisymmetric mapping.
\begin{definition}\label{def:quasisym} Let $\eta:[0,\infty)\to [0,\infty)$ be a homeomorphism.
    A homeomorphism $f\colon X\to Y$ is an \emph{$\eta$-quasisymmetry}, if for all $x,y,z\in X$, with $x\neq z$, 
    \[
    \frac{d_Y(f(x),f(y))}{d_Y(f(x),f(z))} \leq \eta \left(\frac{d_X(x,y)}{d_X(x,z)}\right).
    \]
    We call $\eta$ the \emph{quasisymmetry function} of the mapping $f$.
     
    If such a homeomorphism $f$ exists, we say that $(X,d_X)$ and $(Y,d_Y)$ are quasisymmetric, or $\eta$-quasisymmetric if we wish to emphasize the $\eta$-dependence.
\end{definition}
We write $(X,d_X)\sim_{\QS} (Y,d_Y)$ if $X$ and $Y$ are quasisymmetric and $(X,d_X)\sim_{\QS,\eta} (Y,d_Y)$ if $X,Y$ are $\eta$-quasisymmetric. If $X=Y$, we also write $(X,d_1)\sim_{\QS,\eta,I} (X,d_2)$ if the identity map $I(x)=x$ is a $\eta-$quasisymmetry between $(X,d_1)$ and $(X,d_2)$.

It is an easy observation that the inequality in Definition \ref{def:quasisym} can be reversed (if $x\neq y$) by taking reciprocals and switching the role of $x$ and $y$. Thus, if $f$ is quasisymmetric, then
\begin{equation}\label{eq:invquas}
    \frac{d_Y(f(x),f(y))}{d_Y(f(x),f(z))} \geq \frac{1}{\eta \left(\frac{d_X(x,z)}{d_X(x,y)}\right)}.
\end{equation}
In particular, if $f$ is an $\eta$-quasisymmetry, then $f^{-1}$ is a $\eta_1$-quasisymmetry with $\eta_1(t)=\frac{1}{\eta^{-1}(t^{-1})}$.

We define the collection of all metrics on a given space $X$ that are $\eta$-quasisymmetric to an initial metric $d$:
\[
\cC_{\eta,d} = \{d':X\times X \to [0,\infty) : d' \text{ is a metric and } (X,d) \sim_{QS,\eta,I} (X,d')\}.
\]

The following lemma is trivial, but seems to not have been emphasized in the literature previously. Fix a metric space $(X,d)$, let $D(X)$ be the space of all metrics on $X$; of course, $\cC_{\eta,d} \subset D(X)$. 
\begin{lemma}\label{lem:convexcone}
    The set $\cC_{\eta,d}$ is a closed convex sub-cone of $D(X)$.
\end{lemma}
\begin{proof}
    The cone $\cC_{\eta,d}$ is closed and convex because it is defined by the following linear constrants
    \[
    d'(x,y)\leq \eta \left(\frac{d(x,y)}{d(x,z)}\right) d'(x,z).
    \]
    for each $x,y,z\in X$ with $x\neq z$. 
\end{proof}

One more class of mappings will play a role below:
\begin{definition}
A mapping $f\colon X\rightarrow Y$ between two metric spaces is a \emph{Lipschitz quotient map} if there is a constant $L>0$ such that
$$ B(f(x), L^{-1}r) \subseteq f(B(x,r)) \subseteq B(f(x),Lr) \text{ for all } x\in X, r>0.$$
\end{definition}
Note that the second inclusion just says that $f$ is $L$-Lipschitz. Lipschitz quotient maps need not be homeomorphisms, but are always open mappings.

\subsection{Poincar\'e and Loewner}

Following Keith \cite{Keith}, we use the following notation: $\text{LIP}_0(X)$ is the space of all real-valued Lipschitz functions on $X$ with compact support. For $f\in \text{LIP}_0(X)$, define the pointwise Lipschitz constant to be 
$$ \Lip[f](x) = \limsup_{r\rightarrow 0} \sup_{y\in B(x,r)} \frac{|f(y)-f(x)|}{r}.$$

We will use the basic fact that if $f\colon X\rightarrow Y$ is $L$-Lipschitz and $g\colon Y \rightarrow \RR$ is Lipschitz, then
\begin{equation}\label{eq:Lip}
 \Lip[g\circ f](x) \leq L \cdot \Lip[g](f(x)).
\end{equation}

The following definition of the Poincar\'e inequality in metric measure spaces is different than that originally introduced in \cite{HK}, but equivalent by \cite{Keith}. See also \cite{KRajala}.
\begin{definition}
A complete metric measure space $(X,d,\mu)$ is a \emph{$p$-PI space}, for some $p\in [1,\infty)$, if
\begin{enumerate}[(a)]
\item $\mu$ is a doubling measure, and
\item There is a constant $C>0$ such that
$$ \fint_B |u-u_B| d\mu \leq C \diam(B) \left( \fint_{CB} (\Lip[u])^p\right)^{1/p}$$
whenever $u\in \text{LIP}_0(X)$ and $B$ is a ball in $X$. (Here $\fint_B g$ and $g_B$ both mean the average of $g$ over the set $B$, i.e., $\mu(B)^{-1}\int_B g$.)
\end{enumerate}
A \emph{PI} space is a $p$-PI space for some $p\in[1,\infty)$.

\end{definition}

Closely related to the Poincar\'e inequality is the \emph{Loewner condition} on a metric measure space. Vaguely speaking, this condition requires that each pair of continua in the space is connected by a sufficiently thick family of curves, as measured by modulus. The precise definition of modulus and the Loewner condition play no role in the results below. Only the following crucial fact, due to Heinonen--Koskela \cite{HK}, is needed. (See also \cite[Ch. 9]{He}.) For the purposes of this paper, a reader may take this result to be the definition of an Ahlfors $Q$-regular $Q$-Loewner space.

\begin{theorem}[Heinonen--Koskela]\label{thm:LoewnerPI} Let $Q>1$.
A complete Ahlfors $Q$-regular space is Loewner with exponent $Q$ if and only if it is a $Q$-PI space.
\end{theorem}
We will use the term \emph{$Q$-Loewner} to mean ``Ahlfors $Q$-regular and Loewner with exponent $Q$'' and \emph{Loewner} to mean $Q$-Loewner for some $Q>1$. Although it is possible to discuss the Loewner condition independent of Ahlfors regularity, we will never do so here. We emphasize: \textbf{in this paper, ``Loewner'' means Ahlfors $Q$-regular and Loewner with exponent $Q$, for some $Q>1$.}

A basic property that follows from either the $Q$-PI space or the $Q$-Loewner condition is quasiconvexity. A metric space is \emph{quasiconvex} if each pair of points can be joined by a curve whose length is bounded above by a fixed multiple of the distance between the points. See \cite[Ch. 8]{He} and \cite[Ch. 8]{HKST}.

 The next lemma is well known; see \cite[Lemma 4.3]{Bourdon2002}.   The push-forward of a Borel measure $\mu$ on $X$ by a Borel measurable mapping $f:X\to Y$ is given by $f_*(\mu)(A)=\mu(f^{-1}(A))$. If $\mu,\nu$ are two measures and $\mu$ is absolutely continuous with respect to $\nu$, we write $\mu \ll \nu$.

\begin{lemma}\label{lem:notPI}
Assume that $(X,d,\mu)$ is a PI space and that $f:X\to I$ is surjective and Lipschitz, where $I\subset \RR$ is an interval. Then,  $f_*(\mu) \ll \lambda$, where $\lambda$ denotes Lebesgue measure.
\end{lemma}

We will apply in a few instances the following quasisymmetric invariance of the Loewner property.

\begin{proposition}\label{prop:quasiinv} Assume $Q>1$. Let $(X,d_X)$ and $(Y,d_Y)$ be Ahlfors $Q$-regular and quasisymmetrically equivalent. Then $X$ is $Q$-Loewner if and only if $Y$ is $Q$-Loewner.
\end{proposition}

This proposition was established in \cite[Theorem 8.5]{HK} under a slightly stronger assumption on $X$. Specifically, the authors used a $p$-Poincar\'e inequality for $p<Q$. This assumption was removed by the self-improvement of the Poincar\'e inequality \cite{KZ}, where the main result showed that a $Q$-Poincar\'e inequality already implies a $p$-Poincar\'e inequality for some $p<Q$. An alternate direct proof, which bypasses the need for self-improvement, was discovered by Tyson; see \cite[Corollary 1.6]{TQuasi}.

\subsection{Cheeger's theory and Gromov--Hausdorff blowups}\label{subsec:prelimcheeger}

A deep result of Cheeger \cite{Ch} says that all PI spaces admit a form of differentiable structure. For the following definitions, we follow the phrasing of Bate \cite{Bate}.

If $(X,d,\mu)$ is a metric measure space, then an \emph{$n$-dimensional chart} in $X$ is a pair $(U,\phi)$, where $U\subseteq X$ and $\phi\colon X\rightarrow\mathbb{R}^n$ is Lipschitz. A Lipschitz function $f\colon X \rightarrow\RR$ is said to be \emph{differentiable at $x_0\in U$ with respect to $(U,\phi)$} if there is a unique $Df(x_0)\in\RR^n$ such that
$$ \limsup_{x\rightarrow x_0} \frac{|f(x)-f(x_0)-Df(x_0)\cdot (\phi(x)-\phi(x_0))|}{d(x,x_0)} = 0.$$
(Thus, $f$ looks to first order at $x_0$ like a linear function composed with $\phi$.)

The space $(X,d,\mu)$ is a \emph{Lipschitz differentiability space} if there are countably many charts $(U_i, \phi_i\colon X \rightarrow\RR^{n_i})$ of dimensions $n_i$ such that
$$ \mu\left(X\setminus \cup_i U_i\right) =0$$
and such that every Lipschitz function $f\colon X \rightarrow \RR$ is differentiable at $\mu$-almost every point of every chart $U_i$.

The supremum of the dimensions $n_i$ of the charts is called the \textit{analytic dimension} of the Lipschitz differentiability space. The analytic dimension of a PI space is always at least $1$ \cite[Section 3.3]{KeithLD}.

The main result of \cite{Ch} is the following; see \cite[Theorem 4.38]{Ch}. There are now also a number of alternative proofs or expositions of the result; see \cite{KeithLD},  \cite{KleinerMackay}, \cite[Ch. 13]{HKST}, or \cite[Section 10]{Bate}.
\begin{theorem}[Cheeger]\label{thm:cheeger}
Each PI space is a Lipschitz differentiability space. The analytic dimension of the PI space can be bounded above depending only on the doubling and Poincar\'e constants of the space.
\end{theorem}
In fact, Schioppa showed that the analytic dimension can be bounded by the so called Assouad dimension of the space \cite{Schioppa}. It follows from the resolution of Cheeger's conjecture in \cite{PMR} that the analytic dimension can also be bounded by the Hausdorff dimension of the space. For Ahlfors $Q$-regular spaces, $Q$ is equal to both the Assouad and Hausdorff dimensions.

We require various ``blowup'' arguments in the paper. In most of them, we use a simple ad hoc notion of blowup using the self-similar structure of our example; see Section \ref{sec:blowup}. However, in some places we need the notion of a (pointed, measured) Gromov--Hausdorff blowup, which relies on the notions of pointed measured Gromov--Hausdorff convergence. See \cite[Ch. 8]{DS} or \cite[Ch. 11]{HKST} for definitions and background.

If $(X,d)$ is a metric space and $x\in X$, then a \emph{Gromov--Hausdorff blowup} of $(X,d)$ at $x$ is any pointed Gromov--Hausdorff limit of a sequence of pointed metric spaces of the form $(X, \lambda_n^{-1}d, x)$ where $\lambda_n\rightarrow 0$. If $\mu$ is a measure on $(X,d)$ and in addition the sequence $(X, \lambda_n^{-1}d, \sigma_n^{-1}\mu, x)$ converges in the pointed measured Gromov--Hausdorff sense (for some $\sigma_n\rightarrow 0$) then the limit is a \emph{Gromov--Hausdorff blowup} of $(X,d,\mu)$ at $x$.

If $(X,d,\mu)$ is a PI space, then it admits Gromov--Hausdorff blowups at every point $x\in X$. In fact, if one starts with any given subsequence $\lambda_n$ as above, one can choose the $\sigma_n$ so that a subsequence of $(X, \lambda_n^{-1}d, \sigma_n^{-1}\mu, x)$ converges. Finally, a result of Cheeger \cite[Section 9]{Ch} (see also \cite[Ch. 11]{HKST}) implies that the resulting Gromov--Hausdorff blowups will themselves be PI spaces.

The proof of our main theorem requires a result about PI spaces whose Gromov--Hausdorff blowups are topological planes. It is a consequence of work of Kleiner and the first author \cite[Theorem 1.3 and Remark 1.4]{DavidKleiner}, combined with that of Bate \cite[Theorem 6.6]{Bate}.
\begin{theorem}\label{thm:davidkleiner}
Let $(X,d,\mu)$ be a PI space. Suppose that at $\mu$-a.e. $x\in X$, each Gromov--Hausdorff blowup of $X$ at $x$ is homeomorphic to $\mathbb{R}^2$.

Then the analytic dimension of $X$ is at most $2$. Furthermore, if $U$ is a chart of analytic dimension $2$, then $\mu|_U<<\mathcal{H}^2$ and $\mu|_U$ is $2$-rectifiable, meaning that $U$ is covered up to $\mu$-measure zero by Lipschitz images of subsets of $\RR^2$.
\end{theorem}

\subsection{Metric iterated function systems}

We will need some basic definitions and properties of iterated function systems. For the classical theory, see \cite{falconer,hutchinson} and for the metric version see \cite{stella}.

\begin{definition}\label{def:IFS} Let $(X,d)$ be a complete metric space. A \emph{self-similar iterated function system (IFS)} on $X$ is a collection of mappings $\cF=\{F_1, \dots, F_n:X\to X\}$ that are contractive self-similarities: there exist constants $\lambda_i\in (0,1)$ such that
\[
d(F_i(x),F_i(y))=\lambda_i d(x,y), \forall x,y\in X.
\]

The \emph{attractor} of the IFS is the unique compact set $K\subset X$ for which 
\[
K=\bigcup_{i=1}^n F_i(K).
\]

The \emph{similarity dimension}, $\Sdim(\cF)$, of the IFS is equal to the value $s>0$ for which
\[
\sum_{i=1}^n \lambda_i^s = 1.
\]

We say that $\cF$ satisfies the \emph{open set condition} if there exists a non-empty bounded open set $O$ for which 
\[
F_i(O)\cap F_j(O)=\emptyset \text{ whenever } i\neq j
\]
and $F_i(O)\subset O$ for all $i\in\{1,\dots, n\}$.
\end{definition}

The following theorem is well-known, and was shown in \cite{stella}.
\begin{theorem}\label{thm:IFS} If $(X,d)$ is a complete metric space and $\cF$ is an IFS on $X$, then the attractor $K$ of the IFS exists, is unique, and satisfies $\Hdim(K)\leq \Sdim(\cF)$.

If $\cF$ further satisfies the open set condition, then $\Sdim(K)=\Hdim(K)$. Further, for $p=\Hdim(K)$, there exists a constant $C>0$ for which we have 
\[
C^{-1} r^{p} \leq \HH^p(B(x,r)) \leq C r^p
\]
for all $r \in (0,2\diam(X)].$ In particular, $K$ is Ahlfors $p$-regular.
\end{theorem}

\section{Substitution complexes and the ``pillow space''}\label{sec:cubical}
In this section, our goal is to precisely define the ``pillow space'' $(X,d_0)$ defined informally in the introduction, and to introduce some notation related to this space that will be used in the rest of the paper.

To do so, we develop some terminology about \emph{cubical complexes} and \emph{pullbacks}. While abstract and technical in some aspects, we hope this framework will be useful for future constructions. This approach also has the advantage of explicating the naturality of many of the constructions we do, and gives us a coordinate-independent way of describing the spaces we are working with. Indeed, we will express each stage $X_n$ of the construction as a categorical pullback. This categorical perspective is inspired by the subdivision rules from \cite{CFP}.

The framework for constructing the pillow space is different from \cite[Example 4.3]{AEBb}. We thus end the section by giving an outline of how the results in \cite{AEBb} give a proof of Theorem \ref{thm:AEB}.
\subsection{Cubical complexes}

The stages in our construction will be $2$-dimensional cubical complexes. We fix the following definition for our purposes:

\begin{definition}
A \emph{cubical complex} is a metric space $X$ together with data $(\cC, l,d)$, where $\cC$ is a finite collection of mappings $p:[0,l]^d\to X$. Write $Q_0=[0,l]^d$. We assume the following properties: 
\begin{enumerate}[(i)]
    \item The space $X= \bigcup_{p\in \cC} p(Q_0)$ and each $p$ is a homeomorphism onto its image in $X$.
    \item If  $p(Q_0)\cap p'(Q_0)\neq \emptyset$, then there are faces $S,S'$ of $Q_0$ such that $p(Q_0)\cap p'(Q_0)=p(S)=p'(S')$ and an isometry $f:Q_0\to Q_0$ such that $p'\circ f|_S=p|_S$.
    \item The metric on $X$ is given by
\begin{equation}\label{eq:cubemetric}
d(x,y)=\inf \{\sum_{i=0}^{n-1} |p_i^{-1}(x_i) - p_i^{-1}(x_{i+1})| : x_0 = x, x_n = y, \text{ and } \forall i, p_i\in\cC \text{ with }x_i,x_{i+1}\in p_{i}(Q_0)\}.
\end{equation}

\end{enumerate} 
\end{definition}

In particular, all our cubes are the same dimension, we assume that the images of interiors of distinct cubes are disjoint, and each $p_i(Q_0)$ is isometric to $Q_0$. Thus, each mapping $p\in\cC$ can be identified with its image in $X$; we will often abuse notation and consider $\cC = \{Q=p(Q_0):p\in\cC\}$. (We could also think of $X$ as a CW complex, but we will avoid using that terminology in this paper.) If $W$ is a metric space with a cubical structure, we sometimes denote by $(\cC_W, l_W)$ the data of the cubical structure.

A sequence of points $(x_i)_{i=0}^n$ as in \eqref{eq:cubemetric} is called an \emph{admissible chain} in $X$. Since there are finitely many cubes in the cubical complex, it is clear that any cubical complex $X$ in this sense is compact.

Suppose that $f:X\to Y$ is a continuous mapping between two cubical complexes, $X$ with data $(\cC,l,d)$ and $Y$ with data $(\cC',l,d)$. We say that $f$ is \emph{cubical} if for each $Q\in\cC$, there is a cube $Q'\in\cC'$ such that $f|_Q$ is an isometry onto $Q'$.  

The following lemma is then easy.
\begin{lemma} \label{lem:cubicallemma} Let $X,Y$ be two cubical complexes and let $f\colon X \rightarrow Y$ be a cubical map. Then $f$ is $1$-Lipschitz. If in addition $f$ has a cubical left inverse, then $f$ is an isometry.
\end{lemma}
\begin{proof}
It is immediate from the assumptions that if $(x_i)$ is an admissible chain from $x$ to $y$ in $X$, then $f(x_i)$ is an admissible chain from $f(x)$ to $f(y)$ in $Y$. Since in addition $f$ is an isometry on each cube, we have immediately from the definition of the metric in \eqref{eq:cubemetric} that $f$ is $1$-Lipschitz. 

Next, suppose that a cubical map $g:Y\to X$ is a left inverse of $f$. Let $x,y\in X$. We know that $d(f(x),f(y))\leq d(x,y)$. To show the opposite inequality, we observe that $g$ is $1$-Lipschitz by the first part of the proof, and hence $d(x,y)=d(g(f(x)), g(f(y)))\leq d(f(x),f(y))$.
\end{proof}

Given a cubical complex $X$ with data $(\cC,l,d)$ and $L>0$, we may \emph{expand $X$ by a factor of $L$} simply by expanding each cube by a factor of $L$  (with the metric scaling accordingly). More precisely, we replace the cube $Q_0=[0,l]^d$ by $[0,Ll]^d$, but otherwise use the same collection $\cC$ of characteristic mappings.

We may also wish to subdivide a cubical complex:

\begin{definition}     Given a metric space $X$ together with a cubical complex structure $(\cC,l)$, the \emph{subdivision} of the cubical complex is a cubical complex structure $(\cC^L, l^L\colon=l/L)$ obtained by simply subdividing each cube of $\cC$ into $L^d$ identical sub-cubes.
    
    More precisely, set $Q_0^L = [0,l/L]^d$ and divide $Q_0$ into $L^d$ subsquares isometric to $Q_0^L$.  The characteristic maps $p^L\in \cC^L$ are obtained by restricting each $p\in C$ to one of the copies of $Q_0^L$ in $Q_0$ and pre-composing by the unique translation that gives it domain $Q_0^L$. 
\end{definition}

A subdivision does not alter the metric on the underlying metric given space of the cube complex, but simply the structure of the space as a cubical complex.

\subsection{Substitution complexes}

We now describe a framework for building new cubical complexes out of old ones.

\begin{definition}\label{def:substitutioncomplex} Fix $d\in\mathbb{N}$, $l>0$ and $L\in \N$. A \emph{substitution complex} is a tuple $(X,Y, Z,S, \pi)$, where $X$, $Y$, and $Z$ are cubical complexes of dimension $d$ and side lengths $l$, $l$, and $Ll$, respectively, and $S$ and $\pi$ are mappings that satisfy the following conditions.
\begin{itemize}
\item $S\colon X \to Z$ is continuous, open, surjective, and cubical from the cubical structure on $X$ obtained after expanding by a factor of $L$. (In particular, $S$ is  $L$-Lipschitz in the original metric on $X$ by Lemma \ref{lem:cubicallemma}.) 
\item  $\pi:Y\to Z$  is continuous, surjective, and cubical when $Z$ is equipped with the cubical complex structure subdivided by $L$. (In particular, $\pi$ is $1$-Lipschitz by Lemma \ref{lem:cubicallemma}.)
\end{itemize}
\end{definition}

Given this data, we can form a new topological space as the pullback of the following diagram.

\[\begin{tikzcd}
	X && Y \\
	& Z
	\arrow["S", from=1-1, to=2-2]
	\arrow["\pi"', from=1-3, to=2-2]
\end{tikzcd}\]
Namely, we define a space $P(X,Y)$ together with maps $\tilde{\pi}, \tilde{S}$ as follows.
\[\begin{tikzcd}
	& {P(X,Y)} \\
	{X} && {Y} \\
	& Z
	\arrow["{S}"', from=2-1, to=3-2]
	\arrow["{\pi}", from=2-3, to=3-2]
	\arrow["{\tilde{\pi}}"', from=1-2, to=2-1]
	\arrow["{\tilde{S}}", from=1-2, to=2-3]
\end{tikzcd}\]
The maps $\tilde{\pi}, \tilde{S}$ are called the pullback maps, and the space $P(X,Y)$ is called the \emph{pullback} of the substitution complex. (Although the pullback depends on all five elements of $(X,Y,Z,S,\pi)$, in our applications below typically $Z,S,\pi$ will be clear from context and only $X$ and $Y$ will change, hence the notation $P(X,Y)$.) The pullback can be described topologically as
\[
P(X,Y)=\{(x,y)\in X\times Y: S(x)=\pi(y)\},
\]
and it is equipped with the subspace topology from the product space $X\times Y$. The natural projection maps give our mappings $\tilde{S}(x,y)=y$ and $\tilde{\pi}(x,y)=x$. 

\begin{lemma}\label{lem:cubicallimit} If $(X,Y,Z,S, \pi)$ is a substitution complex, then the pullback $P(X,Y)$ is homeomorphic to a metric space that has the structure of a cubical complex, and  $\tilde{S}$ is cubical and $\tilde{\pi}$ is cubical with respect to the subdivision of $X$ by a factor $L$.
\end{lemma}

\begin{proof} By the definition of a substitution complex, we have constants $l>0$ and $L\in\mathbb{N}$ such that $l=l_{X}=l_Y= l_Z/L$. Consider the following collection of pairs of cubes from $X$ and $Y$:
\[
\mathcal{Q}=\{(p_1,p_2) : p_1 \in \cC_{X}, p_2 \in \cC_{Y}, \text{image}(\pi\circ p_2) \subset \text{image}(S\circ p_1)\}.
\]

If $(p_1,p_2)\in\mathcal{Q}$, write $c_1=\text{image}(p_1)\subseteq X$ and $c_2=\text{image}(p_2)\subseteq Y$ for the associated cubes. Then $\pi(c_2)$ is contained in the cube $S(c_1)$ in $Z$. Since $S$ is cubical, there exists an inverse map for $S$ restricted to $S(c_1)$, denoted by $S^-_{c_1}:S(c_1)\to c_1.$

Let  $l_{P(X,Y)}=l_Z/L^2 = l/L$. Define $p_{p_1, p_2}:[0,l_{P(X,Y)}]^d \to P(X,Y)$ by:
    \[
    p_{p_1, p_2}(a)=(S^-_{c_1}(\pi(p_{2}(L a))), p_{2}(La)).
    \]
and set
$$ \cC_{P(X,Y)} = \{p_{p_1,p_2}: (p_1,p_2)\in\mathcal{Q}\}$$
We need to verify that this indeed gives a cubical structure on $P(X,Y)$. First, each characteristic  map $p_{p_1,p_2}$ is a homeomorphism onto its image, since it is continuous, injective, and defined on a compact set. Second, the cubes cover $P(X,Y)$. Indeed, for each $(x,y)\in P(X,Y)$ we have a cube $c_2=\text{image}(p_2)\in \cC_Y$ so that $y\in c_2$.  Now $\pi(c_2)$ is a cube $c'$ in the subdivision complex of $Z$. It is contained in a cube $c''$ of the original complex of $Z$. Now, $S(x)=\pi(y)\in c''$, and since $S$ is open and cubical,  there exists a cube $c_1=\text{image}(p_1)$ in $X$ so that $x\in c_1$ and $S(c_1)=c''$. We obtain that $(x,y)\in\text{image}(p_{p_1,p_2})$, i.e., that $(x,y)$ is contained in a cube of $\cC_{P(X,Y)}$. 

If two cubes $p_{p_1,p_2}, p_{p_3,p_4}$ from $\cC_{P(X,Y)}$ intersect, then (the images of) $p_1$ and $p_3$ intersect in $X$ and (the images of) $p_2$ and $p_4$ intersect in $Y$. These intersections must be along full faces, and hence the same is true for $p_{p_1,p_2}$ and $p_{p_3,p_4}$ in $P(X,Y)$.

Thus, the dimension $d$, the side length $l_{P(X,Y)}$, and the collection $\cC_{P(X,Y)}$ defines a cubical structure on $P(X,Y)$.  Finally, for each $p_{p_1,p_2}\in\cC(P(X,Y))$, we have $\tilde{S} \circ p_{p_1,p_2}(a) = p_2(La)$ and $\tilde{\pi}\circ p_{p_1,p_2}(a) = p_{2}(S^-_{c_1}(\pi(p_{p_2}(L a))))$ which yields that $\tilde{S}$ is cubical and that $\tilde{\pi}$ is cubical with respect to the subdivision of $X$ (since $S^-_{c_1}$ is such).

\end{proof}

If $X$ and $Y$ are cubical complexes, then we endow $P(X,Y)$ by default with the cubical structure $(\cC_{P(X,Y)}, d, l_{P(X,Y)})$ constructed in the above lemma. Note in particular that by Lemmas \ref{lem:cubicallemma} and \ref{lem:cubicallimit}, the mapping $\tilde{\pi}\colon P(X,Y)\rightarrow X$ is $1$-Lipschitz and surjective.

\subsection{Maps between substitution complexes}

\begin{definition}\label{def:morphism}
    A map between two substitution complexes $(X_i,Y_i,Z_i,S_i,\pi_i)$, $i=1,2$, is a triple of cubical mappings $f_{X}:X_1\to X_2$, $f_{Y}:Y_1\to Y_2$, and $f_Z:Z_1\to Z_2$ such that
    $S_2 \circ f_X = f_Z\circ S_1$ and $\pi_2\circ f_Y = f_Z \circ \pi_1$.

    We say that a map between two substitution complexes is an isomorphism if there is a collection of inverse maps $f_X^{-1}, f_Y^{-1}, f_Z^{-1}$ from $(X_2,Y_2,Z_2,S_2,\pi_2)$ to $(X_1,Y_1,Z_1,S_1,\pi_1)$. If $X_1=X_2, Y_1=Y_2, Z_1=Z_2$, then such a map is an automorphism.

    A map $(g_X,g_Y,g_Z)$ from $(X_2,Y_2,Z_2,S_2,\pi_2)$ to $(X_1,Y_1,Z_1,S_1,\pi_1)$  is called a left inverse of $(f_X,f_Y,f_Z)$ if $g_W \circ f_W$ is the identity for each $W\in\{X,Y,Z\}$.
\end{definition}
\begin{remark}
    It is possible to view a substitution complex as a  functor from the category of diagrams of a given form to the category of cubical complexes. In this language, the map between two substitution complexes would be a natural transformation between two such functors. 
\end{remark}

Given the categorical construction, the following is a natural lemma.

\begin{lemma}\label{lem:morphism} If $(f_X,f_Y,f_Z)$ is a map between two substitution complexes $(X_i,Y_i,Z_i,S_i,\pi_i)$, $i=1,2$, then there is a cubical map $f_{P(X,Y)}:P(X_1,Y_1)\to P(X_2,Y_2)$ for which $\tilde{\pi}_2 \circ f_{P(X,Y)} = f_X \circ \tilde{\pi}_1$ and $\tilde{S}_2 \circ f_{P(X,Y)} = f_Y \circ \tilde{S}_1$, where $\tilde{\pi}_i, \tilde{S}_i$ are the pullback maps.

Furthermore, the mapping $f_{P(X,Y)}$ is $1$-Lipschitz. If the map $(f_X,f_Y,f_Z)$ has a left inverse, then $f_{P(X,Y)}$ is an isometry. If the map is an automorphism, then $f_{P(X,Y)}$ is a self-isometry of $P(X_1,Y_1)$.
\end{lemma}
\begin{proof}
Recalling that
$P(X_i,Y_i)=\{(x,y) \in X_i\times Y_i : S_i(x)=\pi_i(y)\},$
we set 
$$ f_{P(X,Y)}(x,y)=(f_X(x),f_Y(y)).$$
It is direct to verify that this map is continuous and cubical with respect to the cubical structure given in Lemma \ref{lem:cubicallimit}. Thus $f_{P(X,Y)}$ is $1$-Lipschitz by Lemma \ref{lem:cubicallemma}. The rest of the claims follow directly from Lemma \ref{lem:cubicallemma}, since if $(g_X,g_Y,g_Z)$ is a left inverse of $(f_X,f_Y,f_Z)$, then $g_{P(X,Y)}$ is a left inverse of $f_{P(X,Y)}$.
    \end{proof}

\subsection{Iteration of substitution complexes }\label{subsec:iteration}

We now apply the previous construction to a \emph{self-similar substitution complex}, which is a substitution complex in which $X=Y=X_1$ and $Z=X_0$, that is, $(X_1,X_1,Z,S_1,\pi_1)$. Let $L\in\N$ be the subdivision parameter in Definition \ref{def:substitutioncomplex}. At the first stage, we obtain $P(X,Y)=X_2$, and maps $\pi_{2}=\tilde{\pi}_1$ and $S_2=\tilde{S}_1$ from $X_2$ to $X_1$. Since $X=Y$ in this case, we can iterate the construction. Recursively, we take the pullback of the following diagram.
\[\begin{tikzcd}
	X_{n} && X_1 \\
	& Z
	\arrow["S_1 \circ S_n", from=1-1, to=2-2]
	\arrow["\pi_1"', from=1-3, to=2-2]
\end{tikzcd}\]
This defines a space $X_{n+1}$ together with maps $\pi_{n}:X_{n+1}\to X_n$ and $S_{n+1
}:X_{n+1}\to X_1$. Equivalently, we could define $X_n$ as the inverse limit of the following diagram.
\begin{equation}\label{tik:xndiagram}
\begin{tikzcd}
	X_{1} && X_1 && X_1 && X_1 &&\cdots && X_1  \\
	& Z && Z && Z && Z &\cdots& Z  
	\arrow["S_1", from=1-1, to=2-2] \arrow["S_1", from=1-3, to=2-4] \arrow["S_1", from=1-5, to=2-6] \arrow["S_1", from=1-7, to=2-8] \arrow["S_1", from=1-9, to=2-10]
	\arrow["\pi_1"', from=1-3, to=2-2] \arrow["\pi_1"', from=1-5, to=2-4] \arrow["\pi_1"', from=1-7, to=2-6] \arrow["\pi_1"', from=1-7, to=2-6] \arrow["\pi_1"', from=1-9, to=2-8] \arrow["\pi_1"', from=1-11, to=2-10]
\end{tikzcd}
\end{equation}
From this perspective, we can represent the space $X_n$ as
\begin{equation}\label{eq:Xnrep}
X_n=\{(x_1,\dots, x_n) : S_1(x_i)=\pi_1(x_{i+1}), i=1, \dots, n-1\}
\end{equation}
with the topology coming from the product topology. For the purpose of later analysis, we will fix this representation of  $X_n$. We will call the spaces $X_n$ the \textit{replacement complexes} of the substitution complex $(X_1, X_1, Z, S_1, \pi_1)$.

That $X_n$ also has a cubical structure follows from repeated application of Lemma \ref{lem:cubicallimit}. This structure admits the following description, which is a straightforward application of induction and Lemmas \ref{lem:cubicallemma} and \ref{lem:cubicallimit}.

\begin{lemma}\label{lem:cubicalxn}
The space $X_n$ has a cubical structure of the same dimension $d$ as $X_1, X_0$ and side length $l_{X_n}=l_{X_0}/L^n$. The cubes of $X_n$ can be described as 
$$ \cC_{X_n} = \{p_{p_1, \dots, p_n} : p_i \in \cC_{X_1}, \pi_1(c_{i+1})\subset S_1(c_i) \},$$
where $c_i = {\rm image}(p_i)$. The maps $p_{p_1, \dots, p_n}$ are defined by
$$ p_{p_1,\dots, p_n}(a) = (x_1, \dots, x_n),$$
where $x_n = p_n(L^n a)$ and recursively $x_{n-k} = S_{c_{n-k}}^{-1}(\pi(x_{n-k+1})).$
\end{lemma}

The main content of Lemma \ref{lem:cubicalxn} is that each cube $c=\text{image}(p)$ of $X_n$ can be associated to a certain sequence $(c_i=\text{image}(p_i))_{i=1}^n$ of  $n$ cubes from $X_1$. In our main example below, $X_0$ will only have a single cube, and thus the condition $\pi(c_{i+1})\subset S(c_i)$ will hold vacuously. In that case, there is therefore a correspondence between  $\cC_{X_n}$ and the set of \emph{all} sequences $(c_1,c_2,\dots, c_n)$ with $c_i\in \cC_{X_1}$. Thus, $\cC_{X_n}$ can be viewed as the collection of words $w$ of length $n$ in the alphabet $\cC_{X_1}$, i.e., the product set $\cC_{X_1}^n$. 

Finally, we can  define the inverse system of metric measure spaces
\[
X_0 \xleftarrow[\pi_1]{} X_1 \xleftarrow[\pi_2]{} X_2 \xleftarrow[\pi_3]{} X_3 \xleftarrow[\pi_4]{} \cdots \xleftarrow[\pi_n]{} X_n \xleftarrow[\pi_{n+1}]{} \cdots,
.\]
Each $\pi_i$ is a $1$-Lipschitz surjection. For any such sequence of spaces, there is a canonical construction of an inverse limit metric space:
\begin{equation}\label{eq:inverselimit}
X = \{(q_i)_{i=1}^\infty : q_i\in X_i, \pi_i(q_i)=q_{i-1}, i\geq 2\}, \quad d_0((q_i)_{i=1}^\infty,  (q'_i)_{i=1}^\infty)=\lim_{i\to\infty} d_i(q_i,q'_i),
\end{equation}
with $1$-Lipschitz surjections $\pi_{\infty,n}:X\to X_n$ defined by  $\pi_{\infty,n}((q_i)_{i=1}^\infty)=q_n$. We also set $\Pi = \pi_{\infty,0} = \pi_1 \circ \pi_{\infty, 1}.$

These maps are $1$-Lipschitz, surjective, and satisfy the compatibility conditions $\pi_{n} \circ \pi_{\infty,n}=\pi_{\infty,n-1}$. An equivalent way of obtaining $X$ is by taking an inverse limit of an infinite diagram of the form in the Diagram \ref{tik:xndiagram}, whence we can obtain a continuous bijection of $X$ with the set
\begin{equation}\label{eq:sequencespace}
X = \{(x_i)_{i=1}^\infty : x_i\in X_1, \pi_1(x_{i+1})=S_1(x_{i}), i\geq 1\},
\end{equation}
with $q_i$ in representation \eqref{eq:inverselimit} corresponding to $(x_1, x_2, \dots, x_i)$ in representation \eqref{eq:sequencespace}. In representation \eqref{eq:sequencespace}, the mappings $\pi_{\infty,n}$ for $n\geq 1$ are simply
$$ \pi_{\infty,n}((x_i)_{i=1}^\infty) = (x_1, \dots, x_n).$$

In this way we see that $X$ is compact. This space $X$ is called the \emph{inverse limit} of the substitution complex $(X_1,X_1,X_0,\pi,S)$. While the previous discussion gives two representations of $X$, it is \eqref{eq:sequencespace} that we will use in the analysis of the next sections. 

\subsection{Word encoding and self-similar structure}\label{subsec:wordencoding}
Under quite weak conditions, the inverse limit space $X$ is the attractor of an iterated function system. 
 Consider a substitution complex $(X_1,X_1,X_0, \pi, S)$ in the special case where $X_0$ has a single cube. Let $\cC_{X_1}$ be the collection of cubes of $X_1$; as often, we abuse notation and conflate the characteristic maps in $\cC_{X_1}$ with their images. For each $c\in \cC_{X_1}$, we have an inverse map $S_c^{-1}:X_0\to c$ of $S|_{c}:c\to X_0$. The map $S_c^{-1}$ is cubical when the edge length of $X_0$ is divided by $L$. We define $F_c^n:X_n\to X_{n+1}$ by 
\[
F_c^n((x_1,x_2, \dots x_n)) = (S_{c}^{-1}(\pi(x_1)), x_1, x_2, \dots, x_n).
\]
This map is continuous and cubical when the edge length of $X_n$ is divided by $L$. This follows from Lemma \ref{lem:cubicalxn}. This means, by Lemma \ref{lem:cubicallemma}, that $F^n_c$ is $L^{-1}-$Lipschitz. There is also a left inverse to $F^n_c$, the map $\sigma^n:X_{n+1}\to X_n$, which is given by
\[
\sigma_n((x_1,x_2, \dots x_{n+1}))=(x_2,\dots, x_{n+1}).
\]
By again applying Lemma \ref{lem:cubicalxn}, we see that $\sigma_n$ is cubical when the edge length of $X_{n+1}$ is multiplied by $L$. Thus $\sigma_n$ is $L$-Lipschitz.

We may now take an inverse limit of the substitution complex to obtain a space $X$ as in \eqref{eq:sequencespace}. We then obtain maps $F_c$ and $\sigma$ of $X$ as follows:
\[
F_c((x_1,x_2, \dots)) = (S_{c}^{-1}(\pi(x_1)), x_1, x_2, \dots)
\]
and
\[
\sigma((x_1,x_2, \dots ))=(x_2, x_3, \dots).
\]

Using Lemma \ref{lem:cubicallemma} we can show that the mappings $F_c$ are similarity maps. These give further that the inverse limit space $X$ is \textbf{a)} the attractor of an iterated function system, \textbf{b)} which satisfies the open set condition, and \textbf{c)} whose Hausdorff dimension can be computed in terms of the parameter $L$ and the set of cubes $\cC_{X_1}$.

\begin{lemma}\label{lem:IFS} Let $(X_1,X_1,X_0, \pi, S)$ be a substitution complex, $L\in\N, L\geq 2$ be its scaling parameter, and let $X_0$ consist of a single cube. Let $\cC_{X_1}$ be the cubes of $X_1$ and let $F_c$ and $\sigma$ be the maps defined above, for $c\in \cC_{X_1}$. Then the maps $F_c$ are $L^{-1}$ similarities, and $X$ is the attractor of the IFS $\cF=\{F_c : c\in \cC_{X_1}\}$. Moreover, the IFS satisfies the open set condition, and $X$ is Ahlfors $p$-regular with
\[
p={\rm dim}_H(X)=\frac{\log(|\cC_{X_1}|)}{\log(L)}.
\]
\end{lemma}

\begin{proof}
    We have that for every $c\in \cC_{X_1}$ the map $F_c$ is $L^{-1}$-Lipschitz and that $\sigma$ is an $L$-Lipschitz left inverse of the map $F_c$. Thus, $F_c$ is an $L^{-1}$-similarity map. Moreover, from Lemma \ref{lem:cubicalxn}, we can see that
    \[
    X=\bigcup_{c\in \cC_{X_1}} F_c(X).
    \]
    Thus, $X$ is the attractor of the given IFS. Next, let $U$ be the interior of the single cell in $X_0$, and let $O=\{(x_1,\dots): x_1\in U\}$, an open set in $X$. Now, one can see that $F_c(O)=\{(x_1,x_2, \dots) \in X : x_1 \in {\rm Int}(c), x_2\in U\}$. Since the cells $c$ have disjoint interiors in $X_1$, the images $F_c(O)$ are disjoint for distinct $c\in \cC_{X_1}$.  Further, ${\rm Int}(c)\subset U$, so $F_c(O)\subset O$. Thus, the IFS $\cF$ satisfies the open set condition. Therefore, the conclusion of the lemma follows from Theorem \ref{thm:IFS}. 
\end{proof}

We next introduce some natural terminology relating to iterated function systems. 
Given the maps $F_c$ for $c\in \cC_{X_1}$, there is a natural word-encoding for elements in $X$. Namely, for every finite word $w=(c_1, c_2, \dots, c_n) \in \cC_{X_1}^n$ of length $n$ in the alphabet $\cC_{X_1}$, we set $F_w = F_{c_1}\circ F_{c_2} \circ \cdots \circ F_{c_n}$. For the empty word, $w=\emptyset$, we set $F_w$ to be the identity. The collection of all finite length words is denoted $\cC_{X_1}^*=\{\emptyset\} \cup \bigcup_{n=1}^\infty \cC_{X_1}^n.$ The length of a finite word is denoted $|w|$.

Given this, we can define $X^w = F_w(X)$. Further, if $w=(c_1,c_2, \cdots, )\in \cC_{X_1}^\N$ is an infinite word, we can define $F(w)$ to be the unique point
\[
F(w) \in \bigcap_{n=1}^\infty F_{w(n)}(X),
\]
where $w(n) = (c_1, \dots, c_n)$ is the truncation of the word $w$ to its first $n$ letters. This point is unique, since the diameters of the sets $F_{w(n)}(X)$ decrease geometrically in $n\in \N$. Note that the map $w\mapsto F(w)$ is not injective.

\subsection{The pillow space, sheets, and symmetries}\label{subsec:pillowspace}

We now specialize to discuss the specific example that we will study in the rest of the paper.

Our example is the inverse limit of a substitution complex with dimension $d=2$ and subdivision parameter $L=3$. Fix $X_0 = [0,1]^2$, with the obvious cubical structure consisting of a single cube. Let $X_1$ be the cubical complex obtained by subdividing $X_0$ into nine congruent sub-squares and doubling the central square. In other words, we glue to the central square of $X_0$ an additional square of side length $3^{-1}$, attached to the original central square along its boundary. Thus, $X_1$ is a cubical complex with ten cubes (squares). The map $\pi:X_1\to X_0$ is obtained by identifying the glued square with the corresponding points of $X_0$. The map $S:X_1\to X_0$ is obtained by scaling up each sub-square by a factor of $3$, and reflecting in each sub-square appropriately to ensure continuity. 

\begin{definition}\label{def:pillowspace}
The \textit{pillow space} $(X,d_0)$ is the inverse limit of the substitution complex $(X_1, X_1, X_0, S, \pi)$.  
\end{definition}
In particular, $X$ is a compact metric space with a geodesic metric $d_0$.

We next wish to define some important subsets of $X$, which we call ``sheets'', and some related symmetries of $X$. Let $X_1^S$ be the cubical complex obtained by subdividing $X_0$ into nine congruent subsquares, and let $\pi^S\colon X_1^S\to X_0$ be the identity map and $S^S\colon X_1^S\to X_0$ the map that scales the squares by a factor of three and reflects to ensure continuity. This gives a substitution complex $(X_1^S, X_1^S, X_0, S^S, \pi^S)$. Since $\pi^S$ is an isometry, the sequence of replacement complexes $X_n^S$ are all isometric to $X_0$ as metric spaces, with each of the maps $\pi^S_{\infty,n}$ being an isometry. As cubical complexes the spaces $X^S_n$ are simply iterated subdivisions of the unit square $X_0$. We can also identify each point in $X_n^S$ by a sequence $(x_1,\dots, x_n)$, where $x_n =\pi^S(x_n)= S^S(x_{n-1})$. With this identification, $\pi_n^S:X_n^S \to X_0$ is given by $(x_1,\dots, x_n)\mapsto x_1$. The limit $X^S$ of this substitution complex is isometric by $\pi^S_{\infty,0}$ to the unit square $X_0\cong [0,1]^2$. 

We now have cubical maps $(\pi,\pi,{\rm Id})$ from $(X_1,X_1,X_0)$ to $(X_1^S, X_1^S, X_0)$, and thus the map $\pi_n^S: (x_1, \dots, x_n)\mapsto (\pi(x_1), \dots, \pi(x_n))$ is cubical. 

Now, there are (only) two cubical lifts $s_0,s_1:X_1^S\to X_1$ that satisfy ${\rm Id}=\pi\circ s_i$ and $S^S = S \circ s_i$ for $i=0,1$. These correspond to choosing one of the central squares in $X_1$. The maps $s_i$ are isometries; indeed, $\pi$ yields a cubical left inverse (identifying $X_1^S$ with $X_0$). There are also two cubical self-maps $\iota_0, \iota_1:X_1\to X_1$ that satisfy $\pi = \pi \circ \iota_i$ and $S = S \circ \iota_i$ for $i=0,1$. We define $\iota_0$ as the identity map and $\iota_1$ as the map which flips the two central squares, but leaves other squares invariant.

Using these maps, and iterating Lemma \ref{lem:morphism}, we can define the following maps between the sequences of spaces coming from the substitution construction. If $w=(w_1, \dots, w_n)\in \ZZ_2^n$ is a finite word, we define
\begin{align*}
&s_w:X_n^S \to X_n, \ \ \ \ s_w(x_1, \dots, x_n) = (s_{w_1}(x_1), \cdots, s_{w_n}(x_n)) \\
&\iota_w:X_n \to X_n, \ \ \ \ \iota_w(x_1, \dots, x_n) = (\iota_{w_1}(x_1), \cdots, \iota_{w_n}(x_n))
\end{align*}
and the inverse limits of these maps, given an infinite word $w\in \ZZ_2^\N \cong \{0,1\}^\N$, by
\begin{align*}
&s_w:X^S \to X, \ \ \ \ s_w(x_1, \dots) = (s_{w_1}(x_1),s_{w_2}(x_2)\dots) \\
&\iota_w:X \to X, \ \ \ \ \iota_w(x_1, \dots) = (\iota_{w_1}(x_1),\iota_{w_2}(x_2) \dots).
\end{align*}
The maps $s_w, \iota_w$ are all isometries onto their images; recall also that $X_n^S$ and $X^S$ are isometric to $X_0\cong [0,1]^2$. For each infinite word $w\in \ZZ_2^\N$, we call the image of $s_w$ a \emph{sheet} of $X$. For every finite word $w\in \cC_X^n$, the map $s_w$ has a left inverse given by $\pi_n:X_n\to X_n^S$ since ${\rm Id}=\pi \circ s_i$. 
Thus, $s_w$ is an isometry for all finite words $w$.  For an infinite word $w\in \cC_X^\N$, the map $s_w$ is an inverse limit of the isometries $s_{w(n)}$, and thus $s_w$ is an isometry with left inverse given by $\Pi =\pi_{\infty,0}:X_\infty \to X_\infty^S\cong X_0$.

It is easy to see that $\iota_w \circ \iota_{w'}=\iota_{w+w'}$ and hence the map $\ZZ_2^\N \to \{\iota_w:w\in \ZZ^2\}$ yields an isometric action of the compact Abelian group $\ZZ_2^\N$ on $X$. We also see that
\[
\iota_w \circ s_{w'}=s_{w+w'}.
\]
Thus, the isometric action of $\ZZ_2^\N$ transitively permutes the sheets of $X$.

We next study the interactions between these mappings and the mappings $F_c$ and $\sigma$ defined in subsection \ref{subsec:wordencoding}. Let $w,w'\in \ZZ_2^\N$ be infinite words. For $n\in \N$ define $w'|_n=(0,\dots, 0, w'(1), w'(n+1),\cdots)$ to be the word obtained from $w'$ by appending $n$ zeros to the beginning.  Further, by a slight abuse of notation, let $\sigma(w)=(w_2,w_3,\dots)$. It follows directly from the definitions that
\[
\sigma \circ s_w = s_{\sigma(w)} \circ \sigma, \text{ and } \sigma \circ \iota_w = \iota_{\sigma(w)} \circ \sigma.
\]
Further, $\pi \circ \iota_i = \pi$ for $i=0,1$, and thus, by definition of $F_c$, we get
\[
F_c \circ \iota_w = \iota_{w|_1} \circ F_c.
\]
As a consequence and summary of the above calculations, we obtain the following lemma.

\begin{lemma}\label{lem:relationships} Let $w,w'\in \ZZ_2^\N$ be infinite words and let $n\in \N$. We have the following mapping relations 
\[
\iota_w \circ s_{w'}=s_{w+w'},
\]
\[
\sigma \circ s_w = s_{\sigma(w)} \circ \sigma, \sigma \circ \iota_w = \iota_{\sigma(w)} \circ \sigma.
\]
Also, if $w\in \cC_X^n$ and $w'\in \ZZ_2^\N$ we have
\[
F_{w} \circ \iota_{w'} = \iota_{w'|_n} \circ F_{w}.
\]

Moreover, $\Pi \circ s_w$ is the identity, each $\iota_w$ is an isometry, and each $s_w$ is an isometric embedding.
\end{lemma}

Let $p = \frac{\log(10)}{\log(3)}>1$; this will be the Hausdorff dimension of $(X,d_0)$. In fact, we have the following.

\begin{lemma}\label{lem:notloewner}
The space $(X,d_0, \mathcal{H}^p)$ is Ahlfors $p$-regular but is \textbf{not} $p$-Loewner. 
\end{lemma}
\begin{proof} 
The Ahlfors regularity follows immediately from Lemma \ref{lem:IFS}.  

For the second part, recall the projection $\Pi=\pi_{\infty,0}:X\rightarrow \mathbb{R}^2$ and the ``word encoding'' of $X$ from subsection \ref{subsec:wordencoding}. Let $x:X \to \R$ denote the post-composition of $\Pi$ with the $x$-coordinate map from the square to the line. 

Given a triadic interval $I$ in $[0,1]$ of length $3^{-n}$, write $I_{\text{mid}}$ for the middle third of $I$ and $N_I$ for the number of words $w\in \cC_{X_1}^n$ for which $x(X_w)=I$. Because middle squares are doubled in the construction, we have $N_{I_{\text{mid}}}= 4 N_I$ for all triadic $I$.

Now consider the pushforward measure $\mu = x_{*}\mathcal{H}^p$ on $\R$. The similarities of $X$ yield that $\HH^p(X_w)=3^{-p|w|}\HH^p(X)$ for any word $w\in \cC_{X_1}^*$. Thus, for a triadic interval $I$ of length $3^{-n}$, we have $\mu(I) = N_I 3^{-pn} \HH^p(X)$. It then follows from the above that
$$ \mu(I_{\text{mid}}) = \frac{4}{10}\mu(I)$$
for every triadic interval $I$ in $[0,1]$.

It then follows from the Lebesgue differentiation theorem that $\mu$ must be singular to Lebesgue measure on $[0,1]$. Since $x$ is Lipschitz and surjective onto the unit interval, it follows from Lemma \ref{lem:notPI} that it is not possible for $(X,d_0,\mathcal{H}^p)$ to satisfy a Poincar\'e inequality. Since the space is Ahlfors $p$-regular, it can therefore not be $p$-Loewner.
\end{proof}

\subsection{Combinatorial Loewner Property for $X$}
The space $X$ is approximately self-similar and satisfies the Combinatorial Loewner Property. Since we only use this fact, and not the conditions directly, we leave their definitions to \cite{AEBb}. These results follow from \cite{AEBb} after we make some remarks on the relationships between the construction of a pillow space here and there \cite[Example 4.3]{AEBb}.

\begin{proof}[Proof that results from \cite{AEBb} imply Theorem \ref{thm:AEB}] In \cite{AEBb}, the pillow space is constructed using an iterated graph system. The data for this system are:
\begin{enumerate}
    \item An oriented graph $G_1=(V_1,E_1)$.
    \item For each edge $[x,y]\in E_1$ an edge type $t\in \mathcal{T}$.
    \item For each type $t\in \mathcal{T}$ a gluing rule $I_t\subset V_1\times V_1$. 
\end{enumerate}
Using these, a sequence of graphs $G_n$ is constructed recursively as follows. The graph $G_{n+1}=(V_{n+1}, E_{n+1})$, where $V_{n+1}=V_1^{n+1}$ and $[v_1\dots v_{n+1}, w_{1}\dots w_{n+1}]\in E_{n+1}$ is an edge of type $t\in \mathcal{T}$ if and only if either
\begin{enumerate}
    \item $v_1\dots v_{n}=w_1\dots w_{n}$ and $[v_{n+1},w_{n+1}]\in E_1$ is an edge of type $t$; or
    \item $[v_1\dots v_{n},w_1\dots w_{n}]\in E_n$ is an edge of type $t$ and $[v_{n+1},w_{n+1}]\in I_t$.
\end{enumerate}
The graphs $G_n$ are called replacement graphs and a metric space is obtained by taking a limit of re-scaled path metrics.

In our case, $V_1 = \cC_X$ and $E_1$ is the dual graph of $X_1$, that is, two cubes (squares) $c,c'\in V_1$ are connected by an edge if they share a face.  In \cite[Example 4.3]{AEBb}, these squares are indexed by the set $\{(i,j,1): i,j\in \{1,2,3\}\}\cup \{(2,2,2)\}$, where the first two indices describe the Cartesian position in the grid, and the last component described which sheet we are in.  There are two edge types, horizontal ($t=1$) and vertical ($t=2$), depending on if squares intersect vertically or horizontally, and edges are oriented left-right and bottom-top. For the horizontal gluing rule, we set $I_1=\{((3,j,1),(1,j,1)) : j \in \{1,2,3\}\}$ and for vertical edges we set $I_2=\{((j,3,1),(j,1,1)) : j \in \{1,2,3\}\}$. These correspond to gluing the right and left faces, or the top and bottom faces. This replacement rule is a special case of a \emph{cubical replacement rule} defined in \cite[Section 4]{AEBb}.

Given this description, it is a bit technical but rather straightforward to see that the dual graphs of the cubical complexes $X_n$ and the replacement graphs $G_n$ are isomorphic. In \cite{AEBb} the limit space  is constructed as a limit of the metrics $3^{-n}d_{G_n}$. This limit can be taken as a Gromov-Hausdorff limit, or more explicitly using symbolic calculus and a limit (supremum) as in \cite[Definition 7.10]{AEBb}. Further, from the construction of our metric, it is not too difficult to see that the re-scaled path metrics $3^{-n}d_{G_n}$ converge to a space bi-Lipschitz to $X$ and thus identify our limit space $X$ with the limit space $K$ from \cite[Definition 7.10]{AEBb}. From \cite[Theorem 1.1 and Proposition 1.7]{AEBb} we get that $X$ is approximately self-similar and that it satisfies the Combinatorial Loewner Property. We remark, that the crucial geometric feature implying CLP is the dihedral symmetry group, which ensures a type of homogeneity property for discrete moduli. (The bi-Lipschitz deformation comes from the fact that in the graphs $G_n$ the path-metrics correspond to the $\ell_1$-metric on each sheet, but in $X$ we are using the $\ell_2$-metric.) 
\end{proof}

\section{Blowups}\label{sec:blowup}
The goal of this section is to discuss a notion of ``blowup'' that we will use throughout the rest of the paper. 

For this section, fix a substitution complex $(X_1, X_1, X_0, S, \pi)$ and assume that $X_0$ has only a single cube. Let $(X,d_0)$ denote the the inverse limit of this substitution complex, as defined in subsection \ref{subsec:iteration}. We will freely use the notation introduced in Section \ref{sec:cubical} in this section, especially the word encoding for $X$ introduced in subsection \ref{subsec:wordencoding}.

\begin{definition}\label{def:blowup}
Fix $x\in X$, an infinite word $w\in \cC_{X_1}^\N$ for which $F(w)=x$, and a sequence of positive numbers $\lambda_n \rightarrow 0$. Suppose that $d$ is a metric on $X$ (inducing the same topology as $d_0$). Let
$$ d_n(p,q) = \lambda_n^{-1} d(F_{w(n)}(p), F_{w(n)}(q))$$
If $d_n \colon X\times X \rightarrow [0,\infty)$ has a subsequence that converges uniformly to a metric $\hat{d}$ on $X$, then we call $\hat{d}$ a \emph{blowup of $d$ at $x$}.

Similarly, if $\mu$ is a measure on $(X,d)$, $\tau_n\rightarrow 0$, and the sequence
$$ \mu_n := \tau_n^{-1} (\sigma^n_*(\mu|_{X_{w(n)}}))$$
has a subsequential limit in the weak* sense ($\sigma^n$ denoting the $n$-fold iterate of the map $\sigma$), then we call that limit a \emph{blowup of $\mu$ at $x$}.
\end{definition}

Blowup metrics as defined here are somewhat different from (pointed) Gromov--Hausdorff blowups, as discussed in subsection \ref{subsec:prelimcheeger}.
A reader interested in comparing these can verify that our blowups are isometric to \emph{subsets} of the  Gromov--Hausdorff blowups. Gromov--Hausdorff blowups have the undesirable feature that even for self-similar spaces they fail to be isometric to the space itself, since for a compact starting space the Gromov--Hausdorff blowup will typically be unbounded. As we are interested in obtaining metrics on $X$, we take the above approach to blowups. By fixing the parametrizing space $X$ we can also avoid a number of technical issues in the definitions of Gromov--Hausdorff limits (e.g., it is clear with our definition that the topological type does not change).

It is immediate from the self-similarity that each blowup of $(X,d_0)$ is isometric to a rescaling of $(X,d_0)$.  The following lemma characterizing doubling measures on $(X,d_0)$ will be useful.

\begin{lemma}\label{lem:balls}

Consider the metric space $(X,d_0)$. There exists constants $\delta,C>0$ so that the following holds. For every $w\in \cC_{X_1}^*$, there exists an $x\in X_w$ so that
\[
B(x,\delta L^{-|w|})\subset X_w \subset B(x,CL^{-|w|}).
\]
In particular, a measure $\mu$ is doubling if and only if there exists a constant $D>1$ so that for any words $w,v$ with $||w|-|v||\leq 1$ and $d(X_w, X_v)\leq L^{-\max\{|w|,|v|\}}$, we have
\begin{equation}\label{eq:tiledoubling}
D^{-1}\leq \frac{\mu(X_{v})}{\mu(X_w)}\leq D.
\end{equation}
\end{lemma}
\begin{proof} Recall that $X$ satisfies the open set condition with respect to the similarities $\{F_c: c\in\cC_{X_1}\}$. Let $O$ be the open set granted by that condition, and let $B(y,\delta)\subset O$ be any ball inside $O$. We have $B(y,\delta)\subset O \subset X \subset B(y,C)$ for $C=\diam(X)$. Since $F_w$ is a similarity with ratio $L^{-|w|}$, we get 
\[
B(F_w(y),\delta L^{-|w|})=F_w(B(y,\delta))\subset F_w(X) =X_w \subset F_w(B(y,C)) \subset B(F_w(x),CL^{-|w|}).
\]
It follows easily from this that if a measure $\mu$ on $(X,d_0)$ is doubling, then \eqref{eq:tiledoubling} holds.

Conversely, suppose that \eqref{eq:tiledoubling} holds for $\mu$. By Lemma \ref{lem:IFS}, $(X,d_0)$ is Ahlfors regular. Consequently, if $B(x,r)$ is a ball in $X$, then $B(x,2r)$ can intersect only finitely many $X_w$ with $L^{-|w|}\approx r$. On the other hand $B(x,r)$ also contains some $X_v$ with $L^{-|v|}\approx r$. The desired doubling bound follows easily.
\end{proof}

Next, we will study blowups of metrics $d$ that are quasisymmetric to $d_0$ and measures $\mu$ that are doubling.  Some similar results for Gromov--Hausdorff blowups have appeared in \cite{angela,Wenbo,MT}. 

\begin{lemma}\label{lem:blowup}
Suppose that $d$ is a metric on $X$, quasisymmetric (by the identity map) to $d_0$, and $\mu$ is a doubling measure on $(X,d)$. Let $x=F(w)\in X$ for $w\in \cC_{X_1}^\N$. 

\begin{enumerate}[(i)]
\item\label{item:blowupd} There exist a blowup of $(X,d)$ at $x$. In fact, if $\lambda_n\rightarrow 0$ is any sequence, then there is a sequence $k_n\in \mathbb{N}$ such that the metrics
\begin{equation}\label{eq:dn}
d_n(p,q)= \lambda_n^{-1}d(F_{w(k_n)}(p),F_{w(k_n)}(q)) 
\end{equation}
sub-converge to a metric on $X$.

\item\label{item:blowupqs} Every blowup of $(X,d)$ at $x$ is quasisymmetric (by the identity) to $(X,d_0)$. 

\item\label{item:blowupmu}There exists a blowup of $\mu$ at $x$, and every blowup of $\mu$ at $x$ is doubling.

\item\label{item:blowupAR} If $(X,d,\mu)$ is Ahlfors $Q$-regular, then every blow-up of $X$ at $x$ is Ahlfors $Q$-regular.

\item\label{item:blowuploewner} If $(X,d,\mu)$ is $Q$-Loewner, then every blow up of $X$ at $x$ is $Q$-Loewner.
\end{enumerate}
\end{lemma}
\begin{proof} 

Let $d$ be any metric on $X$ that is quasisymmetric to $d_0$, $x\in X$, $w\in \mathcal{C}_{X_1}^\mathbb{N}$ with $F(w)=x$. 

We start with \eqref{item:blowupd}. Up to relabeling sequences, it suffices to prove the ``In fact,...'' statement. We are given a sequence $\lambda_n \rightarrow 0$. Pick any two points $a,b\in X$ with $d_0(a,b)=\diam_{d_0}(X),$ which we may rescale to be $1$ without loss of generality.

By the quasisymmetry, there is a constant $C>0$ and, for each $n$, a radius $\alpha_n>0$ such that
\begin{equation}\label{eq:ballS}
B_{d_0}(x,\alpha_n) \subseteq B_d(x,\lambda_n) \subseteq B_{d_0}(x, C\alpha_n).
\end{equation}
Choose $k_n\in\mathbb{N}$ as small as possible so that $ X_{w(k_n)} \subseteq B_{d_0}(x,\alpha_n)$. 
Observe that
\begin{align*}
d_0(F_{w(k_n)}(a), F_{w(k_n)}(b))&= \diam_{d_0}(X_{w(k_n)})\\
&\approx \alpha_n\\
&\approx \diam_{d_0}(B_d(x,\lambda_n)).
\end{align*}
Therefore, using the quasisymmetry, we have
\begin{equation}\label{eq:dnbound}
d(F_{w(k_n)}(a), F_{w(k_n)}(b))
\approx \diam_d(B_d(x,\lambda_n))
\approx \lambda_n.
\end{equation}

The space $(X,d_0)$ is compact, and thus to find a sublimit of $d_n$, it suffices to prove that $d_n$ are uniformly bounded and equicontinuous. Let $c,d\in X$ be a pair of points with $c\neq d$. We may choose $z\in\{a,b\}$ so that $d_0(z,c)\geq \diam_{d_0}(X)/2$. Thus
\[
\frac{d_0(c,d)}{d_0(c,z)}\leq 2, \ \ \text{  and  } \ \ \frac{d_0(c,z)}{d_0(a,b)}\leq 1.
\]

Then, using \eqref{eq:dnbound},
\begin{align*}
d_n(c,d)&\approx \frac{d(F_{w(k_n)}(c),F_{w(k_n)}(d))}{d(F_{w(k_n)}(a),F_{w(k_n)}(b))}=\frac{d(F_{w(k_n)}(c),F_{w(k_n)}(d))}{d(F_{w(k_n)}(c),F_{w(k_n)}(z))}  \frac{d(F_{w(k_n)}(c),F_{w(k_n)}(z))}{d(F_{w(k_n)}(a),F_{w(k_n)}(b))} \\
&\leq \eta\left(\frac{d_0(F_{w(k_n)}(c),F_{w(k_n)}(d))}{d_0(F_{w(k_n)}(c),F_{w(k_n)}(z))}\right) \eta\left(\frac{d_0(F_{w(k_n)}(c),F_{w(k_n)}(z))}{d_0(F_{w(k_n)}(a),F_{w(k_n)}(b))}\right) \leq \eta(2) \eta(1).
\end{align*}
Thus, $\{d_n\}$ is uniformly bounded. Next, we show that $\{d_n\}$ is uniformly equicontinuous. To show this, fix $\epsilon>0$, and a pair of points $c,d\in X$ with $d_0(c,d)\leq \epsilon$. Again, let $z$ be as above. Then, we have $\frac{d_0(c,d)}{d_0(c,z)}\leq 2 \epsilon$, and arguing as above we get
\[
d_n(c,d)\lesssim \eta(2\epsilon)\eta(1).
\]
From this it is direct to see that the maps $d_n:X\times X\to \RR$ are uniformly equicontinuous. Thus, by passing to a subsequence, we get that $d_n$ converges uniformly to some bounded, non-negative function $
\hat{d}:X\times X\to\RR$. This function is \emph{a priori} a pseudometric. We will next show that it is indeed a metric, and that the metric is quasisymmetric to $d_0$. Let $c,d\in X$ be such that $d_0(c,d)\geq \epsilon \diam_{d_0}(X)>0$. Then, for $z$ chosen as above, we get
\[
\frac{d_0(c,d)}{d_0(c,z)}\geq \epsilon, \text{ and } \frac{d_0(c,z)}{d_0(a,b)}\geq \frac{1}{2}.
\]
Using the inverse quasisymmetry bound \eqref{eq:invquas}, we get
\begin{align*}
d_n(c,d)&\approx\frac{d(F_{w(k_n)}(c),F_{w(k_n)}(d))}{d(F_{w(k_n)}(a),F_{w(k_n)}(b))}=\frac{d(F_{w(k_n)}(c),F_{w(k_n)}(d))}{d(F_{w(k_n)}(c),F_{w(k_n)}(z))}  \frac{d(F_{w(k_n)}(c),F_{w(k_n)}(z))}{d(F_{w(k_n)}(a),F_{w(k_n)}(b))} \\
&\geq \frac{1}{\eta\left(\frac{d_0(F_{w(k_n)}(c),F_{w(k_n)}(z))}{d_0(F_{w(k_n)}(c),F_{w(k_n)}(d))}\right) \eta\left(\frac{d_0(F_{w(k_n)}(a),F_{w(k_n)}(b))}{d_0(F_{w(k_n)}(c),F_{w(k_n)}(z))}\right)} \geq \eta(\epsilon^{-1})^{-1} \eta(2)^{-1}.
\end{align*}

Thus, if $c\neq d$, then $d_n(c,d)$ is bounded below independent of $n$. Thus, $\hat{d}(c,d)>0$. Further, it follows from equicontinuity that the metric $\hat{d}$ is continuous with respect to $d_0$. Thus, since $X$ is compact, we get that the metrics $\hat{d}$ and $d_0$ are homeomorphic. This proves \eqref{item:blowupd}.

It is easy to see \eqref{item:blowupqs}, that each blowup  $\hat{d}$ is quasisymmetric to $d_0$, since the inequality in quasisymmetry \ref{def:quasisym} is stable under uniform convergence of the metrics. 

For \eqref{item:blowupmu}, let $\tau_n=\mu(X_{w(n)})$ and let $\mu_n$ be defined as in Definition \ref{def:blowup}. Thus, $\mu_n$ is a sequence of probability measures on $X$ and possesses a subsequence converging to a blowup measure $\hat{\mu}$ in the weak* sense. We will show that $\hat{\mu}$ is doubling. Let $v_1,v_2\in \cC_{X_1}^*$  be words such that $v_2$ is a prefix of $v_1$. Then, we have by using Lemma \ref{lem:balls} that for all $n\in \N$ that
\[
\frac{\mu_n(X_{v_1})}{\mu_n(X_{v_2})}=\frac{\mu(X_{w(n)v_1})}{\mu(X_{w(n)v_2})}\gtrsim b^{-(|v_1|-|v_2|)}
\]
for some constant $b>0$.  This implies again by Lemma \ref{lem:balls} that $\mu_n$ is uniformly doubling with respect to $d_0$. Any weak limit of uniformly doubling probability measures on $X$ is uniformly doubling with respect to $d_0$. Further, the property of being a doubling measure for a given metric is invariant under quasisymmetric changes of the metric, and thus $\mu$ is also doubling with respect to the metric $\hat{d}$. 

Now assume that $(X,d,\mu)$ is Ahlfors $Q$-regular. Let $v_1\in \cC_{X_1}^*$  be any finite word and let $n\in \N$, we get
\[
\mu_n(X_{v_1})=\frac{\mu(X_{w(n)v_1})}{\mu(X_{w(n)})}\approx \frac{\diam(X_{w(n)v_1})^Q}{\diam(X_{w(n)})^Q} \approx \diam_{d_n}(X_{v_1})^Q.
\]
By doubling, quasisymmetry of $d_n$ with $d_0$ and Lemma \ref{lem:balls}, we can cover any ball $B_{d_n}(x,r)$, with $r\in (0,1)$ by finitely many tiles $X_{v_i}$ whose diameter is comparable to $r$. This yields \eqref{item:blowupAR}, the Ahlfors $Q$-regularity of $(X,\hat{d},\hat{\mu})$, by sending $d\to \infty$.

Finally, assume that $(X,d,\mu)$ is Ahlfors $Q$-regular and $Q$-Loewner. It is quasisymmetric to the Ahlfors $Q$-regular space $(X,\hat{d},\hat{\mu})$. Then, $(X,\hat{d},\hat{\mu})$ is also $Q$-Loewner by Proposition \ref{prop:quasiinv}. This proves \eqref{item:blowuploewner}.

\end{proof}

\section{Symmetrizing Loewner metrics}
We now fix the ``pillow space'' $(X,d_0)$ from Definition \ref{def:pillowspace}. Let $G=\mathbb{Z}_2^\mathbb{N}$, acting on $X$ by the mappings
$$ \iota_w\colon X\rightarrow X \text{ for each }w\in G$$
defined in subsection \ref{subsec:pillowspace}. We equip $G$ with a metric and measure in a standard way, as follows. If $g_1 = (\epsilon_n)$ and $g_2=(\delta_n) \in G$, then we set $D(g_1, g_2)= 3^{-n}$, where $n$ is the minimal integer for which $\epsilon_n\neq \delta_n$. It is standard that $(G,D)$ is a compact metric space. Finally, we define a probability measure $\mathbb{P}$ on $G$ as the product of copies of the uniform measure on the individual $\mathbb{Z}_2$ factors.

Metrics for which all elements of $G$ are isometries play a special role in the arguments below.
\begin{definition}
If $d_S$ is a metric on $X$, then we will say that $d_S$ is \textit{symmetric} if $\iota_w$ is an isometry of $d_S$ for each $w\in G$. 
\end{definition}
The original metric $d_0$ on $X$ is of course an example of a symmetric metric, by Lemma \ref{lem:relationships}.

The goal of this section is to prove Proposition \ref{prop:SL}, which states that if $(X,d_0)$ is quasisymmetric to a $Q$-Loewner metric, then it is quasisymmetric to a symmetric $Q$-Loewner metric. First, we show that symmetric metrics are preserved under blowups:

\begin{lemma}\label{lem:symblowup}
Each blowup of a symmetric metric on $X$ is also a symmetric metric.
\end{lemma}
\begin{proof} Assume that $d_S$ is some symmetric metric on $X$. 
If $w\in \cC_X^\N$ is any infinite word, then 
\[
d_{S,n}(p,q)=\lambda_n^{-1}d_S(F_{w(n)}(p),F_{w(n)}(q))
\]
is a symmetric metric. Indeed, by Lemma \ref{lem:relationships} we have $F_{w(n)}(\iota_{w'}(p))=\iota_{w'|_n}(F_{w(n)}(p))$.  Thus, since $d_S$ is symmetric, we obtain the desired symmetry for $d_{S,n}$:
\[
d_{S,n}(\iota_{w'}p,\iota_{w'}q)=\lambda_n^{-1}d_S(F_{w(n)}(\iota_{w'}p),F_{w(n)}(\iota_{w'}q))=\lambda_n^{-1}d_S(F_{w(n)}(p),F_{w(n)}(q))=d_{S,n}(p,q).
\]
The fact that every blowup of $d_S$ is symmetric now follows immediately by continuity.
\end{proof}

We now work towards the proof of Proposition \ref{prop:SL}. We will find the symmetric Loewner metric $d_{SL}$ by a two-step process. First, we symmetrize the Loewner metric and measure by averaging over $G$. Ideally, this would suffice. However, we also need an additional blow-up step since the symmetrized metric may, as far as we know, fail to be Ahlfors regular.

Let $d_L$ be a $Q$-Loewner metric on $X$ that is quasisymmetric (by the identity map) to $d$. Let $\mu_L$ be the $Q$-dimensional Hausdorff measure with respect to the metric $d_L$, which is Ahlfors $Q$-regular. Note that $Q$ must be the Hausdorff dimension of $(X,d_L)$, and so $Q\geq 2$ as $X$ contains a topological square.

Define a new metric and measure on $X$ by
$$ d_{S}(x,y) = \int_G d_L(gx, gy) d\mathbb{P}(g),$$
$$ \mu_{S}(A) = \int_G \mu_L(g^{-1}(A)) d\mathbb{P}(g).$$

(Recall of course that $g=g^{-1}$ for all $g\in G$, but this seems the natural way to write these expressions.)

\begin{lemma}\label{lem:symmetrize}
We have the following properties of $(X, d_{S}, \mu_{S})$.
\begin{enumerate}[(i)]
\item \label{ds:sym} $d_S$ is symmetric.
\item\label{ds:qs} $d_{S}$ is quasisymmetric (by the identity map) to $d_L$.
\item\label{ds:doubling} $\mu_{S}$ is finite, non-zero, and doubling. 
\item\label{ds:lowerbound} $\mu_{S}(B_{d_S}(x,r)) \gtrsim r^Q$ for all $x\in X$, $r\in (0,\diam_S(X)]$.
\item\label{ds:hausdorff} $\mu_S \gtrsim \mathcal{H}^Q_{d_S}$.
\item\label{ds:positive} $\mathcal{H}^Q_{d_{S}}(B_{d_S}(x,r)) \gtrsim r^Q$ for all $x\in X, r\in (0,\diam_S(X)].$
\end{enumerate}

\end{lemma}
\begin{proof}
Item \eqref{ds:sym} is immediate from the definition of $d_S$.

Observe that for each $g\in G$, 
$$ (x,y)\mapsto d_L(gx,gy)$$
is a metric that is quasisymmetric to $d_0$ with a fixed distortion $\eta$. In other words, each of these metrics is an element of $C_{\eta,d}$. Item \eqref{ds:qs} then follows immediately from Lemma \ref{lem:convexcone}. 

For \eqref{ds:doubling}, the finiteness and non-triviality of $\mu_S$ follow immediately from the Ahlfors regularity of $\mu_L$. For the doubling, note first that, because $d_S$ and $d_L$ are quasisymmetric, there is a constant $C$ such that for each $p\in X, r>0$, there is an $s>0$ so that
$$ B_{d_L}(p, s) \subseteq B_{d_S}(p,r) \subseteq B_{d_S}(p, 2r) \subseteq B_{d_L}(p,Cs).$$ 
Since $\mu_L$ is doubling on $(X,d_L)$,  we have
$$ \mu_L(B_{d_S}(p,2r)) \leq \mu_L(B_{d_L}(p,Cs)) \lesssim \mu_L( B_{d_L}(p, s)) \leq \mu_L(B_{d_S}(p,r)) \text{ for all } p\in X, r>0.$$
Hence, using the fact that each $g\in G$ is an isometry of $(X,d_S)$, we get
\begin{align*}
\mu_S(B_{d_S}(x,2r)) &= \int_G \mu_L(g^{-1}(B_{d_S}(x,2r))\, d\mathbb{P}(g)\\
&= \int_G \mu_L(B_{d_S}(g^{-1}x,2r))\, d\mathbb{P}(g)\\
&\lesssim \int_G \mu_L(B_{d_S}(g^{-1}x,r))\, d\mathbb{P}(g)\\
&= \mu_S(B_{d_S}(x,r)),
\end{align*}
proving \eqref{ds:doubling}

For \eqref{ds:lowerbound}, fix a ball $B=B_{d_S}(x,r)$ (with $r\leq \diam_{d_S}(X)$) and an arbitrary point $y\in B$. Recall that $d_L$ and $d_S$ are quasisymmetric (by \eqref{ds:qs}) and hence each $g\in G$ is a quasisymmetry of $d_L$ with fixed distortion $\eta$. Thus, in the metric $d_L$, $g^{-1}(B)$ is sandwiched between two balls of comparable $d_L$-radius centered at $g^{-1}(x)$, with comparability constant independent of $g$. Since $\mu_L$ is Ahlfors $Q$-regular, we therefore have
$$ \mu_L(g^{-1}(B)) \gtrsim d_L(g^{-1}x, g^{-1}y)^Q \text{ for all } g\in G,$$
with implied constant independent of $g$. Hence, using Jensen's inequality,
\begin{align*}
\mu_S(B) &= \int_G \mu_L(g^{-1}(B))\,d\mathbb{P}(g)\\
&\gtrsim \int_G d_L(g^{-1}x, g^{-1}y)^Q\,d\mathbb{P}(g)\\
&\geq \left(\int_G d_L(g^{-1}x, g^{-1}y)\,d\mathbb{P}(g)\right)^Q\\
&= d_S(x,y)^Q.
\end{align*}
As $y\in B$ was arbitrary and $X$ is connected, we may choose $y$ so that $d_S(x,y)\geq r/2$ and obtain \eqref{ds:lowerbound}.

Item \eqref{ds:hausdorff} follows from \eqref{ds:lowerbound} by a standard covering argument (see,  e.g., \cite[Section 2]{AT}). Lastly, item \eqref{ds:positive} follows from the Loewner property of $d_L$, the quasisymmetry property \eqref{ds:qs}, and \cite[Remark 15.18]{He}.
\end{proof}

If we could prove directly that $\mu_{S}$ were Ahlfors $Q$-regular, then Proposition \ref{prop:SL} would immediately follow from standard results, but unfortunately this does not seem clear. However, we have the following lemma. 

\begin{lemma}\label{lem:density}
The measure $\mu_S$ satisfies
$$ 0 < \liminf_{r\rightarrow 0} \frac{\mu_S(B_{d_S}(x,r))}{r^Q} \leq \limsup_{r\rightarrow 0} \frac{\mu_S(B_{d_S}(x,r))}{r^Q}  < \infty$$
for $\mathcal{H}^Q_{d_S}$-a.e. $x\in X$.
\end{lemma}
\begin{proof}
The first inequality in the lemma holds everywhere by Lemma \ref{lem:symmetrize}\eqref{ds:lowerbound}, and the second is trivial, so we focus on the third.

Lemma \ref{lem:symmetrize} (parts \eqref{ds:doubling} and \eqref{ds:hausdorff}) also implies that $\mathcal{H}^Q_{d_S}(X,d_S)<\infty$. Therefore, there exists a constant $c_Q>0$ for which 
\begin{equation}\label{eq:hausdorfflimsup}
 \limsup_{r\rightarrow 0} \frac{\mathcal{H}^Q_{d_S}(B_{d_S}(x,r))}{r^Q} \leq c_Q
\end{equation}
for $\mathcal{H}^Q_{d_S}$-a.e. $x\in X$. (See \cite[Theorem 6.2]{Mat}, which is stated only in Euclidean space but applies also in this generality.)

By Lemma \ref{lem:symmetrize}\eqref{ds:hausdorff}, we have that $\mathcal{H}^Q_{d_S} \lesssim \mu$, and in particular 
$\mathcal{H}^Q_{d_S} << \mu$. By the Radon-Nikodym and Lebesgue differentiation theorems, we also have
\begin{equation}\label{eq:hausdorffmu}
\lim_{r\rightarrow0} \frac{\mathcal{H}^Q_{d_S}(B(x,r)) }{\mu_S(B(x,r))} \in (0,\infty)
\end{equation}
for $\mu$-a.e. and hence $\mathcal{H}^Q_{d_S}$-a.e. $x\in X$.

Combining \eqref{eq:hausdorfflimsup} and \eqref{eq:hausdorffmu} yields
\[
\limsup_{r\rightarrow 0} \frac{\mu_S(B(x,r))}{r^Q} <\infty
\]
for $\mathcal{H}^Q_{d_S}$-a.e. $x\in X$.
\end{proof}

\begin{proof}[Proof of Proposition \ref{prop:SL}]
We will show that there is a positive $\mu_S$-measure set $A\subseteq (X,d_S, \mu_S)$ such that any blowup of $d_S$ and $\mu_S$ at a point of $A$ yields an Ahlfors $Q$-regular metric measure space $(X,\hat{d},\hat{\mu})$. Such a blowup is quasisymmetric to $(X,d_S)$, and hence $(X,d_0)$, by Lemma \ref{lem:blowup}. By Lemma \ref{lem:symblowup}, it is symmetric. Moreover, since such a blow-up is also quasisymmetric to the Ahlfors $Q$-regular, $Q$-Loewner space $(X,d_L)$, we will have by Proposition \ref{prop:quasiinv} that $(X,\hat{d},\hat{\mu})$ is $Q$-Loewner. 

Thus, we are left to show the existence of the positive measure set $A$. First, let 
\[
\tilde{A}=\{x \in X: 0<\liminf_{r\rightarrow 0} \frac{\mu_S(B(x,r))}{r^Q} \leq \limsup_{r\rightarrow 0} \frac{\mu_S(B(x,r))}{r^Q}<\infty \}.
\]
This set has full  $\HH_{d_S}^Q$-measure by Lemma \ref{lem:density}. In addition, by Lemma \ref{lem:symmetrize}\eqref{ds:hausdorff} and \eqref{ds:positive}, we have $\HH^Q_{d_S}(\tilde{A})>0$ and thus $\mu_S(\tilde{A})>0$. Thus, there is some $\delta>0$ so that the set
\[
A_\delta = \{x\in X: \frac{\mu_S(B(x,r))}{r^Q}\in (\delta, \delta^{-1}) \quad\forall r\in (0,\delta)\}
\]
has positive $\mu_S$-measure. Suppose now that $w$ is a finite word for which $3^{-|w|}\leq \delta$, and $w'$ is any finite word with $|w'|\geq |w|$ and $w'(n)=w$. If $X_{w'}\cap A_\delta \neq \emptyset$, then by the fact that $d_S$ is quasisymmetric to $d_0$, that $\mu_S$ is doubling, and Lemma \ref{lem:balls}, we get
\begin{equation}\label{eq:tilecomp}
\frac{\mu_S(X_w)}{\mu_S(X_{w'})}\approx \frac{\diam_{d_S}(X_{w})^Q}{\diam_{d_S}(X_{w'})^Q}. 
\end{equation}
Let $A\subset A_\delta$ be the set of points $x\in A$ so that for every $\epsilon>0$, there exists a radius $r_\epsilon$ so that for $r\in (0,r_\epsilon)$, we have that $B(x,r)\subset N_{\epsilon r}(A_\delta)$. (That is, the set $A_\delta$ is $\epsilon r$ -dense in $B(x,r)$). By Lebesgue differentiation and the doubling property, the set $A$ has full $\mu_S$-measure in $A_\delta$. 

Let $x\in A$, and let $w\in \cC_X^\N$ be any word for which $F(w)=x$. By Lemma \ref{lem:balls}, for each $x\in A$ and $m\in \N$,  there is an $n_0\in\N$ such that if $n\geq n_0$ and  $w'\in \cC_X^{n+m}$ has $w'(n)=w(n)$, then $X_{w'}\cap A_\delta\neq \emptyset$. Consequently, by \eqref{eq:tilecomp},
\[
\frac{\mu_S(X_{w(n)})}{\mu_S(X_{w'})}\approx \frac{\diam_{d_S}(X_{w(n)})^Q}{\diam_{{d_S}}(X_{w'})^Q},
\]
where the constants are independent of $n,w',m$.  

Now suppose $\hat{\mu},\hat{d}$ are blowups of $\mu_S, d_S$ at $x=F(w)\in A$. Suppose $v\in \cC_{X_1}^*$. Then, taking limits along an appropriate subsequence, we have
\begin{align*}
\frac{\hat{\mu}(X)}{\hat{\mu}(X_v)} &= \lim_{n\rightarrow\infty}\left(\frac{(\sigma^n)_* \mu_S|_{X_{w(n)}}(X)}{(\sigma^n)_* \mu_S|_{X_{w(n)}}(X_v)}\right)\\
&= \lim_{n\rightarrow\infty}\left(\frac{\mu_S(X_{w(n)})}{\mu_S(X_{w(n)v})}\right)\\
&\approx \lim_{n\rightarrow\infty}\left(\frac{\diam_{d_S}(X_{w(n)})^Q}{\diam_{d_S}(X_{w(n)v})^Q}\right)\\
&= \lim_{n\rightarrow\infty}\left(\frac{(\lambda_n^{-1}\diam_{d_S}(F_{w(n)}(X)))^Q}{(\lambda_n^{-1}\diam_{d_S}(F_{w(n)}(X_v)))^Q}\right)\\
&= \frac{\diam_{\hat{d}}(X)^Q}{\diam_{\hat{d}}(X_v)^Q}
\end{align*}
Therefore, $\hat{\mu}(X_v)\approx \diam_{\hat{d}}(X_{v})^Q$ for every $v\in\cC_{X_1}^*$. By Lemma \ref{lem:balls}, and the fact that $\hat{d}$ is quasisymmetric to $d_0$ by Lemma \ref{lem:symblowup}, we get that $\hat{\mu}(B_{\hat{d}}(z,r))\approx r^Q$ for all $r\in (0,\diam_{\hat{d}}(X)]$ and $z\in X$.

\end{proof}

\section{Obtaining a PI structure on the square}

Our goal in this section is to prove Lemma \ref{lem:sheetmetric} and Corollary \ref{cor:sheetPI} below, which say that a symmetric Loewner metric on $X$ induces a PI metric and measure on the unit square.

We first need some additional preliminary facts about sheets. Recall the group $G$ and the mappings $s_w\colon X_0 \rightarrow X$ (for $w\in G$) whose images are the sheets of $X$. From Lemma \ref{lem:relationships}, we get that $\Pi|_{s_w(X_0)}$ is an isometry for each $w\in G$, and that the action of $G$ on $X$ transitively permutes the sheets.

A few other basic facts about sheets are important for us:

\begin{lemma}\label{lem:sheetmove}
Let $d$ be any metric on $X$ quasisymmetric (by the identity) to $d_0$. Let $S$ be a sheet in $X$ and $g\in G$. If $p\in S $ and $x\in g(S)$, then
$$ d(g(p),x) \leq \eta(1) d(p,x),$$
where $\eta$ is the quasisymmetry function of the identity map.
\end{lemma}
\begin{proof}
First, we note that the result holds (with multiplicative constant equal to $1$) if $d=d_0$. Indeed, since $g(p)$ and $x$ are in the same sheet, we have
$$ d_0(g(p), x) = |\Pi(p)-\Pi(x)| \leq d_0(p,x), $$
using that $\Pi$ is $1$-Lipschitz and an isometry on each sheet.

For any other metric $d$ as in the assumptions, we then have by quasisymmetry that
$$ d_0(x,g(p))\leq d_0(x,p) \Rightarrow d(x,g(p)) \leq \eta(1) d(x,p).$$
\end{proof}

\begin{lemma}\label{lem:twosheets}
Let $d_{S}$ be a symmetric, quasiconvex metric on $X$ that is quasisymmetric (by the identity) to $d_0$. Let $S$ and $S'$ be two sheets of $X$, with $p\in S$ and $q\in S'$.

Then there is a point $z\in S\cap S'$ such that
$$ d_{S}(p,z) + d_{S}(z,q) \lesssim d_{S}(p,q),$$
where the implied constant depends only on the quasiconvexity constant and the quasisymmetry function of the identity map.
\end{lemma}
\begin{proof}
To simplify the argument, let us first assume that $\Pi(p) = \Pi(q)$. Let $n$ be the smallest integer such that $\pi_{\infty,n}(p)\neq \pi_{\infty,n}(q)$. It follows that $\pi_{\infty,n}(p)$ and $\pi_{\infty,n}(q)$ are each contained a cube (square) of the cubical complex $X_n$, and these cubes share a common boundary $K_n$. Moreover, $K_n\subseteq \pi_{\infty,n}(S) \cap \pi_{\infty,n}(S')$.

Observe that points in $K_n$ have unique pre-images in $X$ under $\pi_{\infty,n}$. Indeed, if
$$x=(x_1,\dots)\in X \text{ and } \pi_{\infty,n}(x)= (x_1, \dots, x_n) \in K_n,$$
then $x_n$ lies on the boundary of the central doubled square in $X_1$. By definition, $\pi_1(x_{n+1})=S_1(x_n)$, which lies on the outer boundary of $X_1$. There is only one possible $x_{n+1}$ satisfying this condition, namely $S_1(x_n)$ itself, and it lies on the outer boundary of $X_1$. Iterating this argument shows that $x_{n+2}, x_{n+3}, $ etc. and thus $x$ itself are uniquely determined.

Let $K=\pi_{\infty,n}^{-1}(K_n)\subseteq X$. The uniqueness of pre-images discussed above implies that $K\subseteq S\cap S'$. Because the metric $d_S$ is quasiconvex, we may find a curve $\gamma\in X$ with
$$ \ell(\gamma) \lesssim d_S(p,q).$$
The curve $\pi_{\infty,n}(\gamma)$ joins $\pi_{\infty,n}(p)$ to $\pi_{\infty,n}(q)$ and therefore contains a point of $K_n$. It follows that $\gamma$ contains a point $z\in K\subseteq S\cap S'$.

We then have
$$ d_S(p,z) + d_S(z,q) \leq \ell(\gamma) \lesssim d_S(p,q).$$

Now we remove the assumption that $\Pi(p)=\Pi(q)$. As $G$ acts transitively on sheets and $d_S$ is symmetric, we may find an isometry $g\in G$ such that $g(S)=S'$. Using our previous argument, there is a point $z\in S\cap S'$ with
$$ d_S(p,z) + d_S(z,gp) \lesssim d_S(p,gp).$$
It follows from this and Lemma \ref{lem:sheetmove} that
\begin{align*}
d_S(p,z)+d_S(z,q) &\leq d_S(p,z)+d_S(z,gp)+d_S(gp,q)\\ 
&\lesssim d_S(p,gp) + d_S(p,q)\\
&\leq d_S(p,q)+d_S(q,gp) + d_S(p,q)\\
&\lesssim d_S(p,q).
\end{align*}
\end{proof}

\begin{lemma}\label{lem:sheetmetric}
Suppose that $d_{SL}$ is a symmetric Loewner metric on $X$ that is quasisymmetric (by the identity map) to $d_0$. Then there is a metric $\rho$ on the unit square $[0,1]^2$ with the following properties:

\begin{enumerate}[(i)]
\item $\rho$ is quasisymmetric (by the identity map) to the standard Euclidean metric on $[0,1]^2$.
\item The map $\Pi\colon X \rightarrow [0,1]^2$ is a Lipschitz quotient map from $(X,d_{SL})$ to $([0,1]^2,\rho)$.
\item For each sheet $S\subset X$, $\Pi|_S$ is an isometry from $(S,d_{SL})$ to $([0,1]^2,\rho)$.
\end{enumerate}
\end{lemma}
\begin{proof}
Fix a sheet $S_0\subseteq X$, and recall that the map $\Pi|_{S_0}$ is an isometry from $(S_0,d_0)$ to $([0,1]^2, |\cdot|)$. We define the metric $\rho$ on $[0,1]^2$ simply by
$$ \rho(x,y) = d_{SL}( (\Pi|_{S_0})^{-1}(x), (\Pi|_{S_0})^{-1}(y)).$$
It is immediate from the symmetry of $d_{SL}$ and the definition of $\rho$ that $\Pi$ is an isometry from $(S, d_{SL})$ to $([0,1]^2, \rho)$, proving (iii). Item (i) is also immediate: $([0,1]^2, \rho)$ is isometric to $(S_0,d_{SL})$, which is quasisymmetric to $(S_0,d_0)$, which is isometric to $([0,1]^2, |\cdot|)$.

It remains to prove (ii), that $\Pi$ is a Lipschitz quotient map from $(X,d_{SL})$ onto $([0,1]^2,\rho)$. First, we show that $\Pi$ is Lipschitz. Fix $p,q\in X$. Choose a sheet $S$ containing $p$ and an element $g\in G$ such that $g(S)$ contains $q$ (using the fact that $G$ acts transitively on the sheets). We may also choose an element $h\in G$ such that $h(g(S))=S_0$. By Lemma \ref{lem:sheetmove},
$$ \rho(\Pi(p),\Pi(q)) = d_{SL}(hg(p), h(q)) = d_{SL}(g(p),q) \lesssim d_{SL}(p, q).$$

Next we show the ``co-Lipschitz'' condition. Fix any $x\in X$, $r>0$, and $y\in B_\rho(\Pi(x), r) \subseteq [0,1]^2$. Let $S$ be a sheet containing $x$. By the transitivity of the group $G$ on the sheets of $X$, there is an element $g\in G$ such that $g(S) = S_0$. Because $d_{SL}$ is symmetric, $g$ is an isometry of $(X,d_{SL})$. By definition of $\rho$, there is a point $z\in S_0$ such that $\Pi(z)=y$ and
$$ d_{SL}(z,g(x)) = \rho(\Pi(z), \Pi(g(x)))=\rho(y,\Pi(x)) < r.$$

Since $g$ is an isometry that maps $S$ to $S_0$, we have $g^{-1}(z)\in S$ and 
$$ d_{SL}(g^{-1}(z), x) =  d_{SL}(z, g(x)) < r. $$
Moreover, $\Pi(g^{-1}(z)) = \Pi(z)=y$. Therefore, $y\in \Pi(B_{SL}(x,r))$. Since $y$ was arbitrary in $B_\rho(\Pi(x),r))$, we have
$$ B_\rho(\Pi(x),r) \subseteq \Pi(B_{SL}(x,r)),$$
i.e., that $\Pi$ is a Lipschitz quotient mapping.
\end{proof}

\begin{corollary}\label{cor:sheetPI}
Suppose that $d_{SL}$ is a symmetric Loewner metric on $X$ that is quasisymmetric (by the identity map) to $d_0$. Let $\rho$ be the metric on $[0,1]^2$ provided by Lemma \ref{lem:sheetmetric}. Then the space $([0,1]^2,\rho, \Pi_*(\mathcal{H}^Q))$ is a PI space.
\end{corollary}
\begin{proof}
This follows immediately from Lemma \ref{lem:sheetmetric}(ii), Theorem \ref{thm:LoewnerPI}, and Theorem \ref{thm:LQ}.
\end{proof}

\section{Proof of Theorem \ref{thm:main} and Corollary \ref{cor:1dimplane}}\label{sec:mainproof}
In this section, we complete the proof of Theorem \ref{thm:main} (followed by that of Corollary \ref{cor:1dimplane}). The idea is that if $(X,d_0)$ were quasisymmetric to a Loewner space, then that Loewner space could be taken to be symmetric (by Proposition \ref{prop:SL}), and then the planar metric measure space constructed in Lemma \ref{lem:sheetmetric} would be the desired ``analytically one-dimensional square''. The main work lies in the next lemma, which essentially shows that, in this scenario, the planar metric measure space is purely $2$-unrectifiable.

\begin{lemma}\label{lem:unrect}
Suppose that $(X,d_0)$ is quasisymmetric (by the identity) to a symmetric, $Q$-Loewner space $(X,d_{SL})$. Let $([0,1]^2, \rho, \mu=\Pi_*(\mathcal{H}^Q))$ be the associated metric measure space constructed in Lemma \ref{lem:sheetmetric}. 

If $E$ is a compact subset of $([0,1]^2, |\cdot|)$ and $f\colon E \rightarrow ([0,1]^2, \rho)$ is a bi-Lipschitz embedding, then $\mu(f(E))=0$.
\end{lemma}
\begin{proof}
Let us assume, under the conditions of the lemma, that we have such a bi-Lipschitz embedding $f$ of a compact subset $E$ of the standard square into $([0,1]^2, \rho, \mu)$, and suppose that its image $F=f(E)$ has $\mu(F)>0$. Recall that, by Lemma \ref{lem:sheetmetric}, $\Pi\colon (X,d_{SL})\rightarrow ([0,1]^2,\rho)$ is a Lipschitz quotient, and that, by Corollary \ref{cor:sheetPI}, the space $([0,1]^2, \rho, \mu)$ is a PI space.

\begin{claim}\label{claim:sheetbilip}
Under these assumptions, there is a constant $L$ and a blowup metric $\hat{d}_{SL}$ of $d_{SL}$ such that each sheet of $(X,\hat{d}_{SL})$ is $L$-bi-Lipschitz equivalent to the standard Euclidean square $([0,1]^2,|\cdot|)$.
\end{claim}

\begin{proof}
Let $x\in ([0,1]^2,\rho,\mu)$ be a $\mu$-density point of $f(E)$. (See, e.g., \cite[Ch. 1]{He} for the Lebesgue density theorem in this generality.) Let $z$ be any point of $\Pi^{-1}(x)\subseteq X$. Fix a sequence of scales $(\lambda_n)\rightarrow 0$ and associated word $w$ satisfying $F(w)=z$ for which there is a blowup of $d_{SL}$ along $w$ at $z$. (Such a blowup exists by Lemma \ref{lem:blowup}.) We label this blowup metric $\hat{d}_{SL}$. By Definition \ref{def:blowup}, 
\begin{equation}\label{eq:dslblowup}
 \hat{d}_{SL}(p,q) = \lim_{n\rightarrow\infty} d_{SL}^n(p,q)= \lim_{n\rightarrow\infty} \lambda_n^{-1} d_{SL}(F_{w(n)}(p), F_{w(n)}(q)) \text{ uniformly over } p,q\in X.
 \end{equation}
We remind the reader that \eqref{eq:dslblowup} holds only along a subsequence, but to avoid cumbersome notation, we index it as above, considering the limit only along $n$ in the appropriate subsequence. 

The map $f^{-1}$ is bi-Lipschitz from $f(E)\subseteq ([0,1]^2, \rho)$ to the Euclidean square. It therefore has, by McShane's extension theorem, a Lipschitz extension $g$ defined on all of $([0,1]^2,\rho)$. Consider the functions from $X$ to $([0,1]^2, |\cdot|)$ defined by
$$ h_n := \lambda_n^{-1}(g\circ \Pi\circ F_{w(n)}  - g\circ \Pi\circ F_{w(n)}(z)).$$
Observe that, for each $n\in\N$,
\begin{align*}
|h_n(p)-h_n(q)| &= \lambda_n^{-1}|g(\Pi(F_{w(n)}(p))) - g(\Pi(F_{w(n)}(q)))|\\
&\lesssim \lambda_n^{-1}d_{SL}(F_{w(n)}(p), F_{w(n)}(q))\\
&= d_{SL}^n(p,q).
\end{align*}

By a standard Arzel\`a-Ascoli type argument, a subsequence of $\{h_n\}$ therefore converges uniformly in the metric $\hat{d}_{SL}$ to a Lipschitz function $h\colon (X,\hat{d}_{SL}) \rightarrow ([0,1]^2,|\cdot|)$. 

Now consider two points $p,q$ in the same sheet $S$ of $X$. We will next show that 
$$ |h(p)-h(q)| \gtrsim \hat{d}_{SL}(p,q).$$
First, observe that if $y_1, y_2 \in B_\rho(x,r)$, then 
\begin{equation}\label{eq:densitylip}
 |g(y_1) - g(y_2)| \gtrsim \rho(y_1,y_2) - o(r) \text{ as } r\rightarrow 0.
 \end{equation}
This follows from the facts that $x$ is a point of $\mu$-density of $f(E)$, $\mu$ is doubling, and $f^{-1}$ is bi-Lipschitz.

Because $p$ and $q$ are in the same sheet, so are $F_{w(n)}(p)$ and $F_{w(n)}(q)$ for each $n$. By definition of $\hat{d}_{SL}$, we also have
\begin{equation}\label{eq:densitylip2}
\rho(\Pi(F_{w(n)}(p)), x) \leq d_{SL}(F_{w(n)}(p), z) = \lambda_n(\hat{d}_{SL}(p, F_{w(n)}^{-1}(z)) + o(1)) =  O(\lambda_n) \text{ as } n\rightarrow \infty,
\end{equation}
and similarly for $q$. 

Using Lemma \ref{lem:sheetmetric}(iii), \eqref{eq:densitylip}, and \eqref{eq:densitylip2}, it follows that
\begin{align*}
\hat{d}_{SL}(p,q) &= \lim_{n\rightarrow\infty} \lambda_n^{-1} d_{SL}(F_{w(n)}(p), F_{w(n)}(q))\\
&= \lim_{n\rightarrow\infty} \lambda_n^{-1} \rho(\Pi(F_{w(n)}(p)), \Pi(F_{w(n)}(q)))\\
&\lesssim \lim_{n\rightarrow\infty} \lambda_n^{-1}( | g(\Pi(F_{w(n)}(p))) - g(\Pi(F_{w(n)}(q)))| + o(\lambda_n))\\
&= \lim_{n\rightarrow\infty} |h_n(p) - h_n(q)|\\
&= |h(p) - h(q)|
\end{align*}
Thus, $h$ is bi-Lipschitz on each sheet with uniform constant, proving Claim \ref{claim:sheetbilip}.
\end{proof}

We now fix the blowup $(X,\hat{d}_{SL})$ from Claim \ref{claim:sheetbilip}, which (by Lemma \ref{lem:blowup}) is quasisymmetric by the identity map to $(X,d_{SL})$ and hence to $(X,d_0)$. Our next task is to slightly improve the conclusion of Claim \ref{claim:sheetbilip} (after a further blowup) as follows.

\begin{claim}\label{claim:sheetbilip2}
There is a constant $M$ and a blowup metric $\tilde{d}_{SL}$ of $\hat{d}_{SL}$ such that, for each sheet $S\subseteq X$, the restriction of the identity map
$$ \text{id}:(S, \tilde{d}_{SL}) \rightarrow (S, d_0)$$
is $M$-bi-Lipschitz.
\end{claim}
\begin{proof}

Fix a sheet $S_0\subseteq X$. By Claim \ref{claim:sheetbilip}, there is a bi-Lipschitz map 
\[
\psi\colon ([0,1]^2,|\cdot|) \rightarrow (S_0,\hat{d}_{SL}).
\]
Because the identity map from $(X,\hat{d}_{SL})$ to $(X,d_0)$ is quasisymmetric, the map
\[
\psi' = \text{id} \circ \psi \colon ([0,1]^2,|\cdot|) \rightarrow (S_0,d_0)
\]
is quasisymmetric. Recall that $(S_0,d_0)$ is itself isometric to the Euclidean unit square, by Lemma \ref{lem:relationships}.  It thus  follows, by a well-known property of quasisymmetric mappings in Euclidean space, that there is a set $F\subseteq [0,1]^2$ of positive $\mathcal{H}^2$-measure on which the restriction of $\psi'$ is bi-Lipschitz; see, e.g., \cite[Proposition 1.6]{A}. Alternatively, to see this, one could use the fact that planar quasisymmetric maps are quasiconformal and thus they are Sobolev homeomorphisms with a Jacobian that is almost everywhere non-vanishing. Such a Sobolev homeomorphism can be decomposed into bi-Lipschitz parts by the use of classical Lusin-type arguments; see, for example, the argument just before \cite[Lemma 10.7]{HKS}.

Therefore, there is a compact subset $E=\psi(F)\subseteq S_0$ of positive $\mathcal{H}^2_{\hat{d}_{SL}}$-measure such that the identity map from $(E,\hat{d}_{SL})$ to $(E,d_0)$ is bi-Lipschitz.

Let $x\in E$ be a point of $\mathcal{H}^2_{\hat{d}_{SL}}$-density of $E$ within $(S_0, \hat{d}_{SL})$. (Recall that $(S_0,\hat{d}_{SL})$ is bi-Lipschitz equivalent to the Euclidean square, so Lebesgue's density theorem can be applied.) 

If $y_1,y_2 \in S_0 \cap B_{\hat{d}_{SL}}(x,r)$, then each is distance $o(r)$ away from the set $E$, and so
\begin{equation}\label{eq:densitybilip}
 \hat{d}_{SL}(y_1, y_2) \approx d_0(y_1, y_2) +o(r) \text{ as } r\rightarrow 0,
\end{equation}
with implied constant independent of $y_1, y_2$.

Fix a word $w$ with $F(w)=x$, and blow up $\hat{d}_{SL}$ along $w$ with a (sub)sequence of scales $(\eta_n)\rightarrow 0$ to obtain a new metric $\tilde{d}_{SL}$ on $X$. This metric is also symmetric by Lemma \ref{lem:symblowup}. 

Now fix a sheet $S\subseteq X$. Since $F_{w(n)}(S)$ is also a sheet for each $n$, by Lemma \ref{lem:relationships} we may find a (sheet-preserving) isometry $g_n\in G$ such that $g_n(F_{w(n)}(S))=S_0$.

Now fix two arbitrary distinct points $p,q\in S$. Exactly as in \eqref{eq:densitylip2}, we have
$$ \hat{d}_{SL}(F_{w(n)}(p), x) = O(\eta_n) \text{ as } n\rightarrow\infty \text{ along the subsequence},$$
and similarly for $q$. It follows from Lemma \ref{lem:sheetmove} that
\begin{equation}\label{eq:sheetmove}
\hat{d}_{SL}(g_n F_{w(n)}(p), x) = O(\eta_n) \text{ as } n\rightarrow\infty \text{ along the subsequence},
\end{equation}
and similarly for $q$.

Therefore, we have (as $n\rightarrow\infty$ along the chosen subsequence) that
\begin{align*}
\tilde{d}_{SL}(p, q) &=  \eta_n^{-1} \hat{d}_{SL}(F_{w(n)}(p), F_{w(n)}(q)) + o(1) && \text{(definition of blowup)}\\
&= \eta_n^{-1} \hat{d}_{SL}(g_nF_{w(n)}(p), g_nF_{w(n)}(q)) + o(1) && \text{(symmetry of } \hat{d}_{SL}\text{)}\\
&\approx \eta_n^{-1} \left( d_0(g_nF_{w(n)}(p), g_nF_{w(n)}(q)) + o(\eta_n) \right) + o(1) && \text{(\eqref{eq:sheetmove} and \eqref{eq:densitybilip})}\\
&= \eta_n^{-1} d_0(F_{w(n)}(p), F_{w(n)}(q)) + o(1) && \text{(symmetry of } {d}_{0}\text{)}\\
&=\eta_n^{-1} 3^{-n} d_0(p, q) + o(1). && \text{(}F_{w(n)}\text{ scales } d_0 \text{ by } 3^{-n}\text{)}\\
\end{align*}

It follows in particular that the numerical sequence $\{\eta_n^{-1} 3^{-n}\}$ must be bounded above and bounded away from $0$. By passing to a further subsequence, we may therefore assume that this sequence converges to a number $c\in (0,\infty)$. Note that this value is independent of $p$, $q$, and $S$. It then follows from the above calculation that
$$ \tilde{d}_{SL}(p,q) \approx d_0(p,q)$$
when $p,q$ are in the same sheet, proving Claim \ref{claim:sheetbilip2}.
\end{proof}

We now fix the metric $\tilde{d}_{SL}$ constructed in Claim \ref{claim:sheetbilip2}. Recall that this metric is a blowup of $\hat{d}_{SL}$, which was a blowup of $d_{SL}$, which was a $Q$-Loewner metric on $X$ quasisymmetric to $d_0$. By Lemmas \ref{lem:blowup} and \ref{lem:symblowup}, the metric $\tilde{d}_{SL}$ is therefore $Q$-Loewner,  symmetric, and quasisymmetric (by the identity) to $d_0$.

\begin{claim}\label{claim:pillowbilip}
The identity map from $(X,\tilde{d}_{SL})$ to $(X,d_0)$ is bi-Lipschitz.
\end{claim}
\begin{proof}
   
Note that both $\tilde{d}_{SL}$ and $d_0$ are symmetric. In addition, both are quasiconvex metrics on $X$, the former because it is Loewner and the latter because it is geodesic.

Consider any $p,q\in X$. Let $S$ and $S'$ be sheets containing $p$ and $q$, respectively. By Lemma \ref{lem:twosheets}, there is a point $z\in S\cap S'$ such that
$$ \tilde{d}_{SL}(p,z) + \tilde{d}_{SL}(z,q) \lesssim \tilde{d}_{SL}(p,q).$$

Using Claim \ref{claim:sheetbilip2}, we have therefore that
$$\tilde{d}_{SL}(p,q) \gtrsim d_0(p,z)+d_0(z,q) \geq d_0(p,q). $$

The reverse inequality $\tilde{d}_{SL}(p,q) \lesssim d_0(p,q)$ follows by applying the exact same argument beginning with $d_0$ rather than $\tilde{d}_{SL}$.
\end{proof}

Claim \ref{claim:pillowbilip} now leads to a contradiction. The space $(X,\tilde{d}_{SL},\mathcal{H}^Q)$ is a $Q$-Loewner space. If $(X,d_0)$ were bi-Lipschitz equivalent to it, then $Q$ would be the Hausdorff dimension of $(X,d_0)$ and $(X,d_0,\mathcal{H}^Q)$ would be a Loewner space. However, this is not the case by Lemma \ref{lem:notloewner}. This completes the proof of Lemma \ref{lem:unrect}.
\end{proof}

We are now ready to prove Theorem \ref{thm:main}.
\begin{proof}[Proof of Theorem \ref{thm:main}]

Suppose that $(X,d_0)$ were quasisymmetric to a $Q$-Loewner space.  By Proposition \ref{prop:SL}, $(X,d_0)$ would then be quasisymmetric to a symmetric $Q$-Loewner space. Without loss of generality, we may write this Loewner space as $(X,d_{SL})$ with the identity map a quasisymmetry.

Consider the metric measure space $([0,1]^2, \rho, \mu)$ derived from $(X,d_{SL})$ in Lemma \ref{lem:sheetmetric}. By Lemma \ref{lem:sheetmetric} and Corollary \ref{cor:sheetPI}, this is a PI space that is quasisymmetric (by the identity) to the standard square. It remains only to show that this space has analytic dimension $1$.

For each point $p\in [0,1]^2 \setminus \partial ([0,1]^2)$, all the Gromov--Hausdorff blowups of $([0,1]^2,\rho)$ at the point $p$ are homeomorphic (indeed, quasisymmetric) to the plane. This follows by standard compactness results for quasisymmetric mappings; see, e.g., \cite[Lemma 2.4.7]{KL}.

We note also that the boundary of the Euclidean square is porous in the whole square and quasisymmetries preserve porosity, so the set $\partial ([0,1]^2)$ is porous in $([0,1]^2,\rho)$. Since $\mu$ is a doubling measure, it follows that the boundary of the square has $\mu$-measure zero. To summarize the above discussion, $([0,1]^2, \rho,\mu)$ is a PI space, and at almost every point of the space, all blowups are homeomorphic to $\mathbb{R}^2$.

Because $([0,1]^2, \rho,\mu)$ is a PI space, it has an analytic dimension $n\in\N$ by Theorem \ref{thm:cheeger}. By Theorem \ref{thm:davidkleiner}, the dimension $n$ must be either $1$ or $2$. Suppose that $n=2$. By Theorem \ref{thm:davidkleiner}, this forces $\mu|_U$ to be $2$-rectifiable and have $\mu|_U << \mathcal{H}^2$. In particular, we may find a subset $A\subseteq \RR^2$ and a Lipschitz map $f\colon A\rightarrow ([0,1]^2,\rho)$ with $\mu(A)>0$ (and hence $\mathcal{H}^2(A)>0$). By a result of Kirchheim \cite[Lemma 4]{Kirchheim}, we may obtain compact subsets $E_i\subseteq A$ such that $f|_{E_i}$ is bi-Lipschitz for each $i$ and $\mathcal{H}^2(f(A\setminus \cup_i E_i)) = 0$. It follows by the absolute continuity $\mu<<\mathcal{H}^2$ that $\mu(f(A\setminus \cup_i E_i))=0$ and hence that there is some $i$ for which $\mu(f(E_i))>0$.

This contradicts Lemma \ref{lem:unrect}. Therefore, the analytic dimension of the space $([0,1]^2, \rho, \mu)$ is $1$.
\end{proof}

\begin{proof}[Proof of Corollary \ref{cor:1dimplane}]
Suppose that $(X,d_0)$ is quasisymmetric (by the identity, without loss of generality) to a $Q$-Loewner space. By Proposition \ref{prop:SL}, it is quasisymmetric to a symmetric $Q$-Loewner space $(X,d_{SL})$, which induces a PI structure $([0,1]^2, \rho, \mu)$ on the unit square by Lemma \ref{lem:sheetmetric}.

Fix a point $x\in X$ and $y=\Pi(x)\in [0,1]^2$ such that $y$ is not on the boundary of the square. (As noted in the proof of Theorem \ref{thm:main}, this boundary has $\mu$-measure zero.)

The space $([0,1]^2, \rho, \mu)$ has a pointed measured Gromov--Hausdorff blowup at $y$, which we will denote $(P,\hat{\rho},\hat{\mu})$. We claim that $P$ is the desired ``analytically one-dimensional plane''. There are three facts that we would like to verify:
\begin{enumerate}[(i)]
\item $(P,\hat{\rho},\hat{\mu})$ is a PI space,
\item $(P,\hat{\rho})$ is quasisymmetric to the standard $\mathbb{R}^2$, and
\item $(P,\hat{\rho},\hat{\mu})$ has analytic dimension equal to $1$.
\end{enumerate}
The first property is immediate from a result of Cheeger about the persistence of the PI property under pointed measured Gromov--Hausdorff convergence: see \cite[Theorem 9.6]{Ch} or the exposition in \cite[Ch. 11]{HKST}. The second property, as noted in the proof of Theorem \ref{thm:main}, follows from known compactness properties for quasisymmetries, e.g., \cite[Lemma 2.4.7]{KL}.

It remains to show that $(P,\hat{\rho},\hat{\mu})$ has analytic dimension $1$. By Theorem \ref{thm:davidkleiner}, the analytic dimension is either $1$ or $2$; furthermore, if there is a $2$-dimensional chart $U$, then $\mu|_U$ is $2$-rectifiable and absolutely continuous to $\mathcal{H}^2$. Let us suppose that there is such a $2$-dimensional chart. Then, using the result of Kirchheim \cite{Kirchheim} mentioned also during the proof of Theorem \ref{thm:main}, there is subset of $P$ with positive $\hat{\mu}$-measure that is bi-Lipschitz to a subset of $\RR^2$. It follows from standard blowup arguments that there is a point $p\in P$ such that $(P, \hat{\rho})$ has a pointed Gromov--Hausdorff blowup $(\tilde{P}, \tilde{\rho})$ at $p$ that is globally bi-Lipschitz to $\RR^2$. 

By a result of Le Donne \cite[Proposition 3.1]{LD}, the space $(\tilde{P},\tilde{\rho})$ can be realized as a pointed Gromov--Hausdorff blowup of $([0,1]^2, \rho)$ at the point $y$ (if $y$ was chosen outside an exceptional set of $\mu$-measure zero). Thus, there is a sequence $(\lambda_n)\rightarrow 0$ such that
$$ ([0,1]^2, \lambda_n^{-1} \rho, y)$$
converges along a subsequence in the pointed Gromov--Hausdorff sense to the pointed metric space $(\tilde{P}, \tilde{\rho}, \tilde{y})$, for some $\tilde{y}\in \tilde{P}.$

We next wish to find a corresponding blowup of $(X,d_{SL})$ at $x$. Fix an infinite word $w$ with $F(w)=x$. Using Lemma \ref{lem:blowup}\eqref{item:blowupd}, there is a further subsequence such that the metrics
$$ d^n_{SL}(p,q) = \lambda_n^{-1}d_{SL}(F_{w(k_n)}(p), F_{w(k_n)}(q))$$
converge along this subsequence to a blowup metric $\tilde{d}_{SL}$ on $X$, where $k_n$ is some sequence in $\mathbb{N}$. 

By Lemmas \ref{lem:blowup} and \ref{lem:symblowup}, the space $(X,\tilde{d}_{SL})$ is symmetric, $Q$-Loewner, and quasisymmetric to $(X,d_{SL})$ and hence $(X,d_0)$.

Now consider the mappings
$$ \Pi_n = \Pi \circ F_{w(k_n)}: (X,d^n_{SL}) \rightarrow ([0,1]^2, \lambda_n^{-1}\rho).$$
As $\Pi\colon (X,d_{SL})\rightarrow ([0,1]^2, \rho)$ is $1$-Lipschitz and an isometry on each sheet, the same immediately follows for each $\Pi_n$. Thus, passing to a further subsequence, we may assume that the mappings $\Pi_n$ converge to a map $\tilde{\Pi}\colon (X,\tilde{d}_{SL})\rightarrow (\tilde{P},\tilde{\rho})$ that is an isometric embedding when restricted to each sheet of $(X,\tilde{d}_{SL})$. (See, e.g, \cite[Ch. 8]{DS} for an appropriate framework for the convergence of mappings needed here.)

It follows that each sheet of  $(X,\tilde{d}_{SL})$ is bi-Lipschitz equivalent to a subset of $\mathbb{R}^2$. Thus, the metric measure space $([0,1]^2,\tilde{\rho},\tilde{\mu}=\Pi_{*}(\mathcal{H}^Q_{\tilde{d}_{SL}}))$ derived from this space by Lemma \ref{lem:sheetmetric} is bi-Lipschitz equivalent to a subset of $\RR^2$. However, by Lemma \ref{lem:unrect}, this means that 
$$ \mathcal{H}^Q_{\tilde{d}_{SL}}(X) = \tilde{\mu}([0,1]^2) = 0,$$
which contradicts the Ahlfors $Q$-regularity of $(X,\tilde{d}_{SL})$.

The conclusion is therefore that the metric measure space $(P,\hat{\rho},\hat{\mu})$ has no $2$-dimensional chart, and therefore it is a PI space quasisymmetric to $\RR^2$ of analytic dimension $1$.

To complete the proof of Corollary \ref{cor:1dimplane}, it remains to argue that no Gromov--Hausdorff blowup of $P$ is bi-Lipschitz equivalent to $\RR^2$. However, we already showed that the existence of such a blowup leads to a contradiction in the first part of this argument, and hence the proof is complete. 
\end{proof}

\section*{Appendix: Lipschitz quotient maps and PI spaces}

The goal of this section is to prove Theorem \ref{thm:LQ}. As noted in the introduction, this result, and its proof, comes from an earlier unpublished work of Jeff Cheeger and the first author.  Theorem \ref{thm:LQ} is a general result about Lipschitz quotient maps and PI spaces that has nothing to do with the specific setting of our pillow space, hence its placement in this appendix. In fact, we prove a slightly more detailed result that immediately implies Theorem \ref{thm:LQ}.

\begin{definition}
A measure $\mu$ on a metric space $X$ is called \emph{$s$-homogeneous} (for some $s>0$) if there is a constant $C>0$ such that
$$ \frac{\mu(B(x,r))}{\mu(B(x,R))} \geq C^{-1} \left(\frac{r}{R}\right)^s$$
whenever $x\in X$ and $0<r\leq R$.
\end{definition}
Every doubling measure is $s$-homogeneous for some $s>0$; see \cite[p. 103-104]{He}. Recall also that if $(X,d,\mu)$ is a $p$-PI space, then it is also a $p'$-PI space for all $p'>p$ (by H\"older's inequality). Thus, given any PI space, there are finite values of $p$ and $s$ that make it $p$-PI and $s$-homogeneous with $p>s$. Therefore, the following result immediately implies Theorem \ref{thm:LQ}.

\begin{theorem}\label{thm:LQ2}
Let $(X,d,\mu)$ be $p$-PI space. Assume that $\mu$ is $s$-homogeneous with $p>\max(1,s)$. Let $f\colon (X,d)\rightarrow (Y,\rho)$ be a Lipschitz quotient map and $\nu = f_{*}\mu$. If $\nu$ is finite on some ball, then $(Y,\rho,\nu)$  is a $p$-PI space.
\end{theorem}

We now begin working towards the proof of Theorem \ref{thm:LQ}. Our standing assumptions on $f:X\rightarrow Y$ are as follows.
\begin{enumerate}
\item $(X,d,\mu)$ is $p$-PI with $p>s$, $s$ being the homogeneity exponent associated to the doubling measure $\mu$. 
\item $f$ is a Lipschitz quotient map with constant $L$.
\item $\nu=f_* \mu$ is finite on some ball.
\end{enumerate}

\begin{lemma}\label{lem:cover}
Under these assumptions, the following holds for some constant $c\geq 2$ depending only on $L$: For each ball $B=B(y,r)\subset Y$, we have that $f^{-1}(B)$ is contained in a countable union of balls $B_i$ such that
\begin{enumerate}
\item $\text{radius}(B_i) = cr$ for each $i$,
\item $B\subseteq f(B_i)$ for each $i$, and
\item $\frac{1}{c}B_i$ are pairwise disjoint.
\item For each $\lambda>0$, there is a constant $N$ such that the collection $\{\lambda B_i\}$ is ``$N$-overlapping'': each point of $X$ is contained in at most $N$ of the balls $\lambda B_i$. The constant $N$ depends only on $\lambda$, $L$, and the doubling constant of $X$.
\end{enumerate}
\end{lemma}
\begin{proof}
Fix any ball $B=B(y,r)\subseteq Y$. Let $U=f^{-1}(B)\subset X$. Since $X$ is doubling and hence separable, we can cover $U$ by countably many balls $D_i = B(x_i,r)$, with $x_i\in U$. By a basic covering lemma  \cite[Theorem 1.4]{He}, we can extract a disjoint subcollection of $\{D_i\}$ (which we still call $\{D_i\}$) such that $5D_i$ cover $U$. Now, for each $i$, since $x_i\in U$ we have that $f(x_i)\in B$. Therefore,
$$ B \subseteq B(f(x_i), 2r) \subseteq f(B(x_i, 2Lr)) = f(2LD_i)$$
for each $i$. Therefore, if we let $c=\max(2L,5)$ and $B_i = cD_i$, then $B_i$ cover $U$ and $f(B_i)$ contains $B$ for each $i$.

The final condition follows from the doubling property of $X$. Indeed, if $z\in X$ were contained in $n$ such balls $\lambda B_i$, then the centers of those balls would form an $(r/c)$-separated set (i.e. their pairwise distances are at least $r/c$) in $B(z,c\lambda r)$, and hence the number $N$ of such balls is controlled. 
\end{proof}

\begin{lemma}\label{lem:doubling}
Under these assumptions, $\nu=f_*\mu$ is doubling.
\end{lemma}
\begin{proof}
Fix any ball $B=B(y,r)$ in $Y$. Let $B_i$ cover $U=f^{-1}(B)$ as in Lemma \ref{lem:cover}. Observe that for each $i$, there is a sub-ball $B'_i\subseteq 2B_i$ of radius $\frac1L r$ that maps into $B$. (Simply center $B'_i$ at a preimage of $y$ in $B_i$.) Furthermore, since $2B_i$ are $N$-overlapping (for some constant $N$ depending only on $L$ and the data of $X$) and $B'_i \subset 2B_i$, we see that $B'_i$ are $N$-overlapping. From the size and location of $B'_i$ it follows that $\mu(B'_i) \gtrsim \mu(B_i)$, with implied constant depending only on $L$ and the data of $X$.

In addition, if we set $K = \frac{c+2L}{c}$, then the Lipschitz quotient property implies that  $f^{-1}(2B) \supseteq \cup_i (KB_i)$. (A point of $f^{-1}(2B)$ must be at most $2Lr$ from a point of $U$ and hence at most $Kcr$ from a center of some $B_i$.)

Thus, using the doubling property of $\mu$ and the $N$-overlapping property of $B'_i$, we get that
$$ \nu(2B) \leq \sum_i \mu(KB_i) \lesssim \sum_i \mu(B_i) \lesssim \sum_i \mu(B'_i) \lesssim \mu(f^{-1}(B)) =  \nu(B).$$
\end{proof}

We now note that, since $\nu$ is finite on some ball, doubling and non-zero, it is finite and non-zero on every ball. Next, we move on to proving the $p$-Poincar\'e inequality for $Y$. We will use the ``supercritical'' assumption $p>s$ here via the following result, which is contained in a theorem of Haj\l asz-Koskela:
\begin{theorem}[\cite{HajlaszKoskela}, Theorem 5.1]\label{thm:HK}
Let $(X,\mu)$ be a $s$-homogeneous $p$-PI space with PI constant $C>0$ and $p>s$. Let $u$ be a Lipschitz function on $X$ and $B=B(x,r)$ a ball in $X$. Then
$$ \sup_B |u-u_B| \leq Cr\left(\fint_{5CB} \Lip[u])^p d\mu\right)^{1/p}.$$
As an immediate consequence,
$$ \sup_{z,w\in B}  |u(z)-u(w)| \leq 2Cr\left(\fint_{5CB} \Lip[u]^p d\mu\right)^{1/p}.$$
\end{theorem}

\begin{proof}[Proof of Theorem \ref{thm:LQ}]
Throughout this proof, we use the notation $\lesssim$ to hide constants that depend only on $L$ and the data of $X$.

By Lemma \ref{lem:doubling}, we know that $\nu=f_*\mu$ is doubling on $Y$. Fix a ball $B=B(y_0,r)$ in $Y$ and a function $g\in \text{LIP}_0(Y)$. Cover $f^{-1}(B)$ by balls $B_i=B(z_i,cr)$ as in Lemma \ref{lem:cover}, where $c$ depends only on $L$. There is a constant $N$, depending only on $L$ and the data of $X$, such that the collection of balls $\{5CB_i\}$ are $N$-overlapping, where $C$ is the Poincar\'e inequality constant of $X$.

We first observe a certain comparability of measures, which follows from the doubling property of $\nu$, the $N$-overlapping properties of $\{5CB_i\}$, and the fact that $$f^{-1}(B) \subset \bigcup B_i \subset \bigcup 5CB_i \subset f^{-1}(10CcLB) = f^{-1}(C'B).$$
From these, we get that
\begin{equation}\label{eq:measurebounds}
\nu(B) \leq \sum_i \mu(B_i) \leq \sum_i \mu(5CB_i) \leq N \mu\left(\bigcup_i CB_i\right) \leq N\nu(C'B) \lesssim \nu(B).
\end{equation}

We now proceed through a chain of inequalities. Following the chain, we justify each line in the calculation.
\begin{align}
\fint_B |g - g_B| d\nu &\leq 2\fint_B |g - g(y_0)|d\nu\label{triangle}\\ 
&=\frac{1}{\nu(B)} \int_{f^{-1}(B)} |g(f(x)) - g(y_0) | d\mu(x)\label{pushforward}\\
&\leq \frac{1}{\nu(B)} \sum_i \int_{B_i} |g(f(x)) - g(y_0) | d\mu(x)\label{covering}\\
&= \frac{1}{\nu(B)} \sum_i \int_{B_i} |g(f(x)) - g(f(x_i))| d\mu(x) \hspace{0.2in}\text{ where } x_i\in B_i, f(x_i) = y_0 \label{localsurjection}\\
&\lesssim \frac{1}{\nu(B)} \sum_i \mu(B_i)(cr)\left(\fint_{5C B_i} \Lip[g\circ f](x)^p d\mu(x)\right)^{1/p} \label{PIuse}\\
&\lesssim \frac{1}{\nu(B)} \sum_i \mu(B_i)(cr)\left(\fint_{5C B_i} \Lip[g](f(x))^p d\mu(x)\right)^{1/p} \label{Lip}\\
&\lesssim \frac{1}{\sum_i \mu(B_i)} \sum_i \mu(B_i)(cr)\left(\fint_{5C B_i} \Lip[g](f(x))^p d\mu(x)\right)^{1/p} \label{nudoubling1}\\
&\lesssim r  \left(\frac{1}{\sum_i \mu(B_i)} \sum_i \mu(B_i)\frac{1}{\mu(5CB_i)}\int_{5CB_i} \Lip[g](f(x))^p d\mu(x)\right)^{1/p} \label{Jensen}\\
&\lesssim r \left(\frac{1}{\sum_i \mu(B_i)}\sum_i  \int_{5CB_i} \Lip[g](f(x))^p d\mu(x)\right)^{1/p} \label{mudoubling}\\
&\lesssim r  \left(\frac{1}{\sum_i \mu(B_i)} \int_{f^{-1}(C'B)} \Lip[g](f(x))^p d\mu(x)\right)^{1/p}\label{bigpreimage}\\
&\lesssim r \left(\frac{1}{\nu(C'B)} \int_{f^{-1}(C'B)} \Lip[g](f(x))^p d\mu(x)\right)^{1/p}\label{nudoubling2}\\
&= r  \left(\frac{1}{\nu(C'B)} \int_{C'B} \Lip[g](y)^p d\nu(y)\right)^{1/p}\label{lastline}
\end{align}
Line \eqref{triangle} is a simple consequence of the triangle inequality.\\
Line \eqref{pushforward} follows from $\nu = f_*\mu$.\\ 
Line \eqref{covering} holds because $\bigcup B_i$ contains $f^{-1}(B)$.\\ 
Line \eqref{localsurjection} follows from the fact that $B\subseteq f(B_i)$ for each $i$. \\ Line \eqref{PIuse} follows from Theorem \ref{thm:HK} and the fact that $B_i$ has radius $cr$.\\ 
Line \eqref{Lip} follows from \eqref{eq:Lip} and the fact that $f$ is $L$-Lipschitz.\\ 
Line \eqref{nudoubling1} follows from \eqref{eq:measurebounds}.\\
Line \eqref{Jensen} follows from Jensen's inequality, i.e. the concavity of the function $t\mapsto t^{1/p}$, in the form that $\frac{\sum a_i (b_i)^{1/p}}{\sum_i a_i} \leq \left(\frac{\sum a_i b_i}{\sum_i a_i}\right)^{1/p} $.\\ 
Line \eqref{mudoubling} holds simply because $\mu(B_i)\leq \mu(CB_i)$.\\ 
Line \eqref{bigpreimage} follows from the fact that $\{5CB_i\}$ are $N$-overlapping and contained in $f^{-1}(C'B)$, as noted before \eqref{eq:measurebounds}.\\ 
Line \eqref{nudoubling2} again follows from \eqref{eq:measurebounds}.\\Finally, line \eqref{lastline} follows again  from the fact that $\nu=f_*\mu$.

\end{proof}

\begin{remark}
A very simple-minded example illustrating Theorem \ref{thm:LQ} is the following: Let $X=[0,1]^2$ with the usual metric $|\cdot|$ and measure $\mathcal{L}^2$, which is of course a PI space. Let $Y=[0,1]$ with the standard metric. Let $f:X\rightarrow Y$ be the projection to the $x$-axis. Then $f$ is a Lipschitz quotient map. Of course in this case $(Y,f_{*}\mathcal{L}^2)$ is simply $[0,1]$ with Lebesgue measure, which is again a PI space.

This example indicates that such mappings can decrease the analytic dimension (in this case from $2$ to $1$). It is not very difficult to show that Lipschitz quotient mappings cannot \textit{increase} the analytic dimension.
\end{remark}

\bibliographystyle{acm}
\bibliography{clp}

\end{document}